\documentclass[12pt]{amsart}
\usepackage{amssymb}
\usepackage[all]{xy}
\usepackage{fullpage}
\usepackage[usenames]{color}
\usepackage{etex,hyperref}
\usepackage{mathrsfs}
\usepackage{cleveref}
\usepackage{paralist}

\newtheorem{hypothesis}{Hypothesis}

\newcommand{\AF}{{\bar A}}
\newcommand{\AK}{{\hat A}}
\newcommand{\AO}{A}
\newcommand{\BO}{B}
\newcommand{\Ext}{\operatorname{\sf Ext}}
\renewcommand{\ge}{\geqslant}
\newcommand{\Hom}{\operatorname{\sf Hom}}
\newcommand{\JA}{J_\AO}
\newcommand{\JQ}{J_Q}
\newcommand{\Ker}{\operatorname{\sf Ker}}
\renewcommand{\le}{\leqslant}
\newcommand{\cO}{{\mathcal{O}}}
\newcommand{\Rad}{\operatorname{\sf Rad}}
\newcommand{\rank}{\operatorname{\sf rank}}
\newcommand{\Tor}{\operatorname{\sf Tor}}
\newcommand{\R}{\mathcal{R}}
\renewcommand{\Re}{\mathbb{R}}
\newcommand{\C}{\mathbb{C}}
\newcommand{\Des}{\mathscr{D}}
\newcommand{\Char}{\mathrm{char}}
\newcommand{\ccl}{\mathscr{C}}
\newcommand{\D}{\mathrm{D}}
\newcommand{\Q}{\mathbb{Q}}
\newcommand{\sym}[1]{\mathfrak{S}_{#1}}
\renewcommand{\mod}{\mathrm{mod}\ }
\newcommand{\N}{\mathrm{N}}
\newcommand{\un}[1]{\underline{#1}}
\newcommand{\Fix}{\mathrm{Fix}}
\newcommand{\im}{\mathrm{im}\hspace{.08cm}}
\newcommand{\B}{\Omega}
\newcommand{\Ba}{\mathcal{B}}
\newcommand{\lcm}{\operatorname{\sf lcm}}
\newcommand{\Sp}{\mathrm{span}}
\newcommand{\A}{\mathcal{A}}
\newcommand{\Fa}{\mathcal{F}}

\newcommand{\supp}{\mathrm{supp}}
\newcommand{\Ed}{\mathcal{E}}
\renewcommand{\k}{k}
\renewcommand{\gcd}{\operatorname{\sf gcd}}
\renewcommand{\det}{\operatorname{\sf det}}
\newcommand{\KD}{{\hat \Des}}
\newcommand{\FD}{{\bar \Des}}
\newcommand{\OD}{{\Des}}
\newcommand{\Nil}{\mathscr{N}}
\newcommand{\ON}{{\Nil}}
\newcommand{\FN}{{\bar \Nil}}
\newcommand{\KN}{{\hat \Nil}}
\newcommand{\OB}{{\cO B}}
\newcommand{\KB}{{\K B}}
\newcommand{\FB}{{\F B}}
\newcommand{\OM}{{\cO M}}
\newcommand{\KM}{{\K M}}
\newcommand{\FM}{{\F M}}
\newcommand{\I}{{\mathsf{J}}}
\newcommand{\ifa}{0_\Fa}
\newcommand{\ied}{0_\Ed}

\newcommand{\U}{\mathsf{U}}
\newcommand{\T}{\mathsf{T}}
\newcommand{\J}{\mathscr{J}}
\newcommand{\Hec}{\mathscr{H}}
\newcommand{\Rleq}{\leq}

\newcommand{\cB}{{\mathcal{B}}}
\newcommand{\F}{{\mathbb{F}}}
\newcommand{\K}{\mathbb{K}}
\newcommand{\Z}{{\mathbb{Z}}}
\newcommand{\fm}{{\mathfrak{m}}}

\numberwithin{equation}{section}
\newtheorem{theorem}[equation]{Theorem}
\newtheorem{lemma}[equation]{Lemma}
\newtheorem{corollary}[equation]{Corollary}
\newtheorem{proposition}[equation]{Proposition}
\theoremstyle{definition}
\newtheorem{example}[equation]{Example}

\newtheorem{notation}[equation]{Notation}
\newtheorem{remark}[equation]{Remark}
\newtheorem{definition}[equation]{Definition}
\newtheorem{question}[equation]{Question}

\newcommand{\KJ}[1]{\textcolor{red}{#1}}
\newcommand{\KJQ}[1]{\textcolor{blue}{#1}}

\title[Projective Modules and Cohomology for Integral Basic Algebras]{Projective Modules and Cohomology for Integral Basic Algebras}
\author{David J. Benson}
\address[D. J. Benson]{Institute of Mathematics, University of Aberdeen, Aberdeen AB24 3UE,
United Kingdom.}
\email{d.j.benson@abdn.ac.uk}

\author{Kay Jin Lim}
\address[K. J. Lim]{Division of Mathematical Sciences, Nanyang Technological University, SPMS-04-01, 21 Nanyang Link, Singapore 637371.}
\email{limkj@ntu.edu.sg}

\begin{document}

\thanks{The second author would like to thank the Edinburgh Napier University for their support at the beginning for this project.}

\begin{abstract} Algebras defined over fields of characteristic zero and positive characteristic usually do not behave the same way. However, for certain algebras, for example the group algebras, they behave the same way as the characteristic zero case at ``good enough" prime. In this paper, we initiate the study of this topic by imposing increasingly strong hypotheses on basic algebras. When the algebras satisfy the right hypotheses, we have equalities of the dimensions of their cohomology groups between simple modules and  equalities of graded Cartan numbers. The examples include the Solomon descent algebras of finite Coxeter groups at large enough primes, nilCoxeter algebra, and certain finite semigroup algebras at an arbitrary prime.
\end{abstract}

\maketitle

\section{Introduction}\label{S:Intro}

In the representation theory of finite groups,
characteristic zero and ``good enough" prime characteristic
behave the same way. Namely in both cases, the
group algebra is semisimple. If the fields are algebraically
closed then the simple modules correspond and have
the same dimensions. Here, the characteristic is good
enough if it does not divide the group order.

One may wonder whether a similar situation holds for
classes of algebras that are not semisimple in
characteristic zero. A general setup might involve an
algebra defined over a suitable ring of integers, where
one can either pass to the field of fractions, or reduce
modulo a maximal ideal, and then ask
what it might mean for a maximal ideal to be good enough for the algebra.

There are various questions that arise here, and correspondingly
there are various notions of ``good enough".
For example, one might at least want the simple modules to
correspond, and have the same dimensions. And then,
how similar are the cohomology algebras? Of course, they
cannot be isomorphic, as they have different characteristics.
But at least we can compare their dimensions in each degree.
To refine this, we could examine the product structure and secondary
operations, define various constants in terms of these, and
compare them.
We can also compare the Loewy layers of their projective
indecomposable modules, and other more subtle invariants
of this sort.

Our purpose in this paper is to begin an investigation of this
topic by examining the case where, over the field of fractions,
we have a finite-dimensional basic algebra. We present a series of increasingly strong hypotheses on the algebra. It turns out that for the equality of dimensions
of cohomology groups between simple modules, a fairly weak hypothesis suffices.
The equality of multiplicities of simple modules in Loewy layers of
projective indecomposables requires a much stronger
hypothesis, and we give examples to show that the weaker
hypotheses are not sufficient for this.

The examples of basic algebras we examine in detailed are
the Solomon descent algebra of a finite Coxeter group \cite{Atkinson:1992a,
Atkinson/Pfeiffer/vanWilligenburg:2002a,
Atkinson/vanWilligenburg:1997a,
Bergeron/Bergeron/Howlett/Taylor:1992a,
Bishop/Pfeiffer:2013a,
Blessenohl/Laue:1996a,
Blessenohl/Laue:2002a,
Bonnafe/Pfeiffer:2008a,
Douglass/Pfeiffer/Roehrle:2014a,
Garsia/Reutenauer:1989a,
Lim:2023a,
Pfeiffer:2009a,
Saliola:2008a,
Schocker:2004a,
Solomon:1976a}, the nilCoxeter algebra \cite{Bernstein/Gelfand/Gelfand:1973a,Fomin/Stanley:1994a,Khovanov:2001a}, and certain
classes of finite semigroup algebras \cite{Berg/Bergeron/Bhargava/Saliola:2011a,Bidigare:1997a,Margolis/Saliola/Steinberg:2021a,Nijholt/Rink/Schwenker:2021a,Saliola:2009a,Schocker:2008a}. The descent algebra was originally constructed by Solomon as a non-commutative analogue of the Mackey's multiplication formula for permutation characters of the parabolic subgroups of the Coxeter group. But it can also be obtained as the opposite algebra of the fixed-point space of a certain face algebra of hyperplane arrangement in real space associated with the Coxeter group. The face semigroup algebra belongs to a broader class known as unital left regular band algebras, which, in turn, are part of the even larger class known as the $\mathcal{R}$-trivial monoid algebras. Notably, this largest class also includes the $0$-Hecke algebras.

We show that the strongest of our hypotheses holds for the descent algebra for sufficiently large primes and the nilCoxeter algebra for an arbitrary prime.
In the case of the semigroup
algebras that we investigate, we show that the
weaker hypotheses hold under suitable conditions,
but we leave open the question of when the stronger one does. While our results do not give an explicit description of the cohomology groups nor the multiplicities of simple modules in Loewy layers of
projective indecomposables for these classes of algebras, they offer a general explanation why these notions behave the same way in characteristic zero and at ``good enough" prime.

In preparation for this task, in the next section, we give a general discussion of various
hypotheses (Hypotheses \ref{hyp:Ext}--\ref{hyp:freeJ}) on a finitely generated algebra over a local principal
ideal domain, and what they imply. The main results are summarised in \Cref{th:main}. \Cref{se:quiver,S:Ext,S:ExtFree,se:Ofree} are devoted to the proof of the theorem. Then we apply the results to the
cases of the descent algebra of a finite Coxeter
group, the nilCoxeter algebra, and certain semigroup algebras and obtain \Cref{th:descent,T:nilCox,T:R-Trivial,T:RegularBand}. We begin the general discussion on the descent algebra in \Cref{S:Des} and extend the result of the second author regarding the idempotents of the descent algebra to all types and all primes in \Cref{S:Idem}. In \Cref{S: proj}, we show that the descent algebra satisfies our strongest of our hypotheses at large enough prime. The analysis is based on a certain characteristic free basis of the descent algebra and we further investigate that in \Cref{S: indep}. Another algebra satisfying our strongest hypothesis is the nilCoxeter algebra and that is discussed in \Cref{S:nilCox}. In \Cref{S:R-Trivial,S:LRegularBand}, we show that certain classes of semigroup algebras satisfy our weaker hypotheses.

\section{The hypotheses and main results}

Throughout the paper, we shall use the following notation.

\begin{notation}\label{not:A}\
\begin{enumerate}[(i)]
\item Let $\cO$ be a local principal ideal domain with maximal ideal $(\pi)$.
Write $\K$ for the field of fractions of $\cO$ and $\F$ for the residue
field $\cO/(\pi)$ of characteristic $p$.
\item Let $\AO$ be an $\cO$-torsion free $\cO$-algebra, and
let $\AF=\F\otimes_\cO \AO$, $\AK=\K\otimes_\cO A$.
\item Let $M$ and $N$ be $\cO$-free $\AO$-modules, and set
\[ \bar M=\F\otimes_\cO M,\
  \bar N=\F\otimes_\cO N,\
  \hat M=\K\otimes_\cO M,\
  \hat N=\K\otimes_\cO N. \]
\item We write $\Rad(R)$ for the Jacobson radical of a ring $R$, and we
write $\JA$ for $\AO\cap\Rad(\AK)$. By convention, $\JA^0=A$.
\item We denote a general field $\k$ and assume that it has characteristic $q$ (either $q=0$ or $q>0$).
\end{enumerate}
\end{notation}

\begin{hypothesis}\label{hyp:Ext}
  If $M$ and $N$ are
  finitely generated $\cO$-free $\AO$-modules with $\hat N$ simple
  then $\Ext^t_A(M,N)$ is $\cO$-free for all $t\ge  0$.
\end{hypothesis}

\begin{hypothesis}\label{hyp:rad}
Suppose that  the algebra $\AO$ has finite rank over $\cO$, we have
\[ \Rad(\AO)=\pi\AO+\JA, \]
and orthogonal idempotent decompositions of the identity in $\AF$ lift to $\AO$.
\end{hypothesis}


\begin{hypothesis}\label{hyp:freeJ2}
We have $\AK$ is basic and $A$, $A/\JA$ and $\JA/\JA^{\,2}$ are all $\cO$-free.
\end{hypothesis}

\begin{hypothesis}\label{hyp:basic}
Let $Q$ be a finite quiver, and let $\cO Q$ be its quiver
algebra. This is an $\cO$-free $\cO$-algebra with basis the directed
paths in $Q$. The paths of length zero are idempotent. The paths of
length at least one form a two-sided ideal $\JQ$ in $\cO Q$.
Let $I$ be a two-sided ideal of $\cO Q$ such that there exists
$n\ge 2$ with $\JQ^n\le I\le \JQ^2$, and with the property
that the quotient ring $\AO=\cO Q/I$ is $\cO$-free of finite rank.
\end{hypothesis}


\begin{notation}
If \Cref{hyp:basic} holds, we write
\begin{gather*}
\hat \JQ=\K\otimes_\cO \JQ,\qquad \bar\JQ=\F\otimes_\cO \JQ,\\
\hat I=\K\otimes_\cO I\le \K Q,\qquad\bar I=\F\otimes_\cO I \le \F Q,\\
\JA=\JQ/I\le\AO,\quad \hat \JA=
\K\otimes_\cO \JA,\quad \bar\JA=\F\otimes_\cO \JA.
\end{gather*}
We refer to the images  in $\AO$, $\AK$ and
$\AF$ of the paths of length zero in $\cO Q$ as the vertex
idempotents. Furthermore, there is a set $S_1,\dots,S_m$ of $\AO$-modules that are $\cO$-free of
rank one, corresponding to the vertex idempotents. For $1\le i\le m$,
we define $\hat S_i=\K\otimes_\cO S_i$ as an $\AK$-module, and
$\bar S_i=\F\otimes_\cO S_i$ as an $\AF$-module. Then
$\hat S_1,\dots,\hat S_m$ is a complete set of simple $\AK$-modules,
and $\bar S_1,\dots,\bar S_m$ is a complete set of simple
$\AF$-modules.  We write $P_i$ for the projective cover of $S_i$,
namely the image of the vertex idempotent on $\AO$.
Then $\hat P_i=\K\otimes_\cO P_i$ is the projective cover of $\hat
S_i$ over $\AK$, and $\bar P_i=\F\otimes_\cO P_i$ is the projective
cover of $\bar S_i$ over $\AF$.
\end{notation}

\begin{hypothesis}\label{hyp:freeJ}
\Cref{hyp:basic} holds, and for all $n\ge 1$, $\AO/\JA^n$
is $\cO$-free.
\end{hypothesis}

We investigate the implications of the above-mentioned hypotheses and study the Ext groups and radical layers of the projective indecomposable modules. In particular, we compare these under the change of field. The main result of this paper is given as follows.

\begin{theorem}\label{th:main}\
  \begin{enumerate}[\rm (i)]
 \item We have \Cref{hyp:freeJ} $\Rightarrow$
  \Cref{hyp:basic} $\Leftrightarrow$
  \Cref{hyp:freeJ2} $\Rightarrow$
  \Cref{hyp:rad} $\Rightarrow$
  \Cref{hyp:Ext}.
\item Let $t\ge  0$ and $M,N$ be finitely generated $\cO$-free $A$-modules with $\hat N$ simple. If \Cref{hyp:Ext} holds for $A$ (so that $\Ext^t_A(M,N)$ is $\cO$-free), then we have
 \begin{align*}
\F\otimes_\cO\Ext^t_\AO(M,N) &\cong \Ext^t_\AF(\bar M,\bar N), \\
\K \otimes_\cO \Ext^t_\AO(M,N) &\cong \Ext^t_\AK(\hat M,\hat N).
\end{align*} Furthermore, $\dim_\F \Ext^t_\AF(\bar M,\bar
N)=\dim_\K\Ext^t_\AK(\hat M,\hat N)$.
\item \Cref{hyp:freeJ} is equivalent to $\JA^n/\JA^{n+1}$ are $\cO$-free for all $n\ge  0$. In this case, the radical layer multiplicities of $\hat P_i$ and $\bar P_i$ are equal, i.e., for all $1\le i,j\le m$ and $n\ge  0$, \[ [\Rad^n(\hat P_i)/\Rad^{n+1}(\hat P_i):\hat S_j]=
[\Rad^n(\bar P_i)/\Rad^{n+1}(\bar P_i):\bar S_j].\]
\end{enumerate}
\end{theorem}

The implication \Cref{hyp:freeJ} $\Rightarrow$
\Cref{hyp:basic} is clear.
\Cref{hyp:basic} $\Rightarrow$
\Cref{hyp:rad} is proved in \Cref{co:basic},
\Cref{hyp:basic} $\Leftrightarrow$ \Cref{hyp:freeJ2} is proved in
\Cref{th:freeJ2<=>basic},
and \Cref{hyp:rad} $\Rightarrow$ \Cref{hyp:Ext}
is proved in Theorem~\ref{th:hyp=>free}. The consequence of
\Cref{hyp:Ext} is proved in  \Cref{th:KOF}, and the
consequence of \Cref{hyp:freeJ} is proved in Theorem~\ref{th:freeJ}.

\section{Quivers and relations over $\cO$}\label{se:quiver}

\begin{theorem}\label{th:basic}
Suppose that $\AO$, $\AK$ and $\AF$
satisfy \Cref{hyp:basic}.
\begin{enumerate}[\rm (i)]
\item We have $\hat \JA=\Rad(\hat A)$ and $\JA=A\cap\hat \JA$.
\item $\JA$ is a nilpotent ideal in $\AO$, and $\AO/\JA$ is a direct product
of copies of $\cO$ as an algebra, spanned by the vertex idempotents,
which are primitive.
\item Idempotents in $\AF$ lift to idempotents in $\AO$.
\item We have $\dim_\K\AK=\dim_\F\AF$.
\item The ideal $\hat I\le\AK$ satisfies $\hat
  J_Q^n\le\hat I\le \hat J_Q^2$. The ideal $\hat \JA$ is the radical of $\AK$, and
$\AK/\hat \JA$ is a direct product of copies of $\K$ as an algebra,
spanned by the vertex idempotents, which are primitive.
\item The ideal $\bar I\le\AF$ satisfies $\bar J_Q^n\le
\bar I\le \bar J_Q^2$. The ideal $\bar \JA$ is the radical of $\AF$, and
$\AF/\bar \JA$ is a direct product of copies of $\F$ as an algebra,
spanned by the vertex idempotents, which are primitive.
\item We have $\Rad(A)=\pi A+\JA=\pi A + (A\cap\Rad(\AK))$.
\item Both $A/\JA$ and $\JA/\JA^2$ are $\cO$-free.
\end{enumerate}
\end{theorem}
\begin{proof}
(i) Let $\AO_0$ be the $\cO$-subalgebra of $\AO$ spanned by the
vertex idempotents, and similarly $\AK_0$ in $\AK$. Then as an
$\cO$-module we have $\AO=\AO_0\oplus \JA$, $\AK=\AK_0\oplus\hat \JA$.
So $\JA=A\cap\hat \JA$.

(ii) Since $J_Q^n\le I$, we have $\JA^n=0$. The quotient $\AO/\JA$ is isomorphic to $\cO
Q/J_Q$, which is a direct product of copies of $\cO$ spanned by the
vertex idempotents, which are primitive.

(iii) Every idempotent in $\AF$ is conjugate to a sum of vertex
idempotents, and therefore lifts to $\AO$.

(iv) This follows from the fact that $\AO$ is $\cO$-free.

(v) Using part (ii), we see that $\hat \JA$ is nilpotent, and $\AK/\hat
\JA=\K\otimes_\cO (\AO/\JA)$ is
isomorphic to $\K Q/\hat J_Q$, which is a direct product of copies of
$\K$, so $\hat \JA=\Rad(\AK)$.

(vi) Again using part (ii), we see that $\bar \JA$ is nilpotent, and
$\AF/\bar\JA=\F\otimes_\cO(\AO/\JA)$ is isomorphic to $\F Q/\bar J_Q$,
which is a direct product of copies of $\F$, so $\bar \JA=\Rad(\AF)$.

(vii) If $\fm$ is a maximal left ideal of $A$ then since $1\not\in\fm$
we have $\fm\cap\cO\le(\pi)$, where $\cO$ is regarded as embedded in
$\AO$ as multiples of the identity. Letting $\bar\fm$ be the image of $\fm$
in $\AF$, we therefore have $\bar\fm\cap\F=0$. Thus $\fm+(\pi)$ is a
proper left ideal in $A$, and so by maximality $(\pi)\le\fm$. It
follows that $(\pi)\le\Rad(A)$. Since $\JA$ is nilpotent, we also have
$\JA\le\Rad(A)$. Finally, $\AO/(\JA+(\pi))\cong\AF/\bar \JA$ is semisimple
by part (vi), and so $\Rad(A)=\JA+(\pi)$.

(viii) Since $I\le \JQ^2$, we have $\JA=\JQ/I$ and $A/\JA\cong (\cO Q/I)/(\JQ/I)\cong \cO Q/\JQ$ and $\JA/\JA^2\cong\JQ/\JQ^2$. Therefore, both $A/\JA$ and $\JA/\JA^2$ are $\cO$-free.
\end{proof}

\begin{corollary}\label{co:basic}
\Cref{hyp:basic} implies \Cref{hyp:rad}.
\end{corollary}
\begin{proof}
This follows from parts (iii) and (vii) of Theorem~\ref{th:basic}.
\end{proof}


\begin{example}
Here are a couple of examples that do not satisfy
\Cref{hyp:basic}. The first is $\AO=\cO[x]/(x^2-x^3)$. In this
case, $Q$ has just one vertex, it has one loop corresponding to $x$, but $I=(x^2-x^3)$
contains no power of $J_Q=(x)\le\cO Q$. So $\JA$ is not a nilpotent
ideal in $\AO$.
The second example is $\AO=\cO[x]/(\pi x^2, x^3)$. Then $x^2\ne 0$ but
$\pi x^2=0$ in $\AO$, so $\AO$ is not $\cO$-free. In this example we have
$\dim_\K\AK=2$, $\dim_\F\AF=3$, so part (v) of Theorem~\ref{th:basic}
does not hold.
\end{example}

\begin{theorem}\label{th:freeJ2<=>basic}
\Cref{hyp:freeJ2} is equivalent to \Cref{hyp:basic}
\end{theorem}

\begin{proof} By \Cref{th:basic}\,(viii), \Cref{hyp:basic} implies
\Cref{hyp:freeJ2}. Conversely, suppose that \Cref{hyp:freeJ2}
holds. Since $\cO$ is a PID and $A$ is $\cO$-free, both $\JA$ and $\JA^2$ are $\cO$-free. Let $B_2$ be an $\cO$-basis for $\JA^2$. Given that $\JA/\JA^2$ is $\cO$-free, the map $\JA\to \JA/\JA^2$ splits. Let $B_1$ be a lift of a basis for $\JA/\JA^2$ so that $B_1\cup B_2$ is a basis for $\JA$. Similarly, together with the idempotent lifting, let $B_0$ be a lift of a basis for $A/\JA$ consisting of idempotents. So $B=B_0\cup B_1\cup B_2$ is an $\cO$-basis for $\AO$. As such, $\hat B=\{1\otimes b:b\in B\}$ forms a basis for $\AK$ and $\{1\otimes b:b\in B_0\}$ is a complete set of primitive orthogonal idempotents for $\AK$. Let $Q$ be the Ext quiver of the basic algebra $\AK$ where the vertices are labelled by $B_0$. Define the algebra homomorphism $\psi:\cO Q\to A$ by $\psi(e_b)=1\otimes b$ for each $b\in B_0$ and, for $b,b'\in B_0$, if $(1\otimes b')(\JA/\JA^2)(1\otimes b)$ has a basis $\{1\otimes (b'xb):x\in S\}$ for some subset $S\subseteq B_1$, mapping the arrows from $b$ to $b'$ bijectively onto $S$. In particular, $\psi(\JQ)\subseteq \JA$. By the construction, $\psi$ is surjective with the kernel $I$ such that $I\le \JQ^2$. Suppose that $\Rad^n(\AK)=0$. We have $\psi(\JQ^n)\subseteq \JA^n=0$, i.e., $\JQ^n\le I$. Thus $\AK$ satisfies \Cref{hyp:basic}.
\end{proof}

\section{$\cO$-freeness of $\Ext$}\label{S:Ext}

In this section, we investigate algebras satisfying \Cref{hyp:rad}. In particular, we prove that they satisfy \Cref{hyp:Ext}.


\begin{lemma}
If $\AO$ satisfies \Cref{hyp:rad}, then $\AO$, $\AK$ and $\AF$ are
semiperfect. So finitely generated projective modules over these rings satisfy
the Krull--Schmidt theorem, and furthermore, all finitely generated modules
have minimal projective resolutions, unique up
to non-unique isomorphism.
\end{lemma}
\begin{proof}
 Since $\pi A \subseteq \Rad (A)$, $A/\Rad (A)$ is a finite dimensional
 semisimple $\F$-algebra and idempotents lift to $A$, it follows that $A$ is semiperfect.
See Lam~\cite[Chapter 8]{Lam:2001a}, especially Definition~23.1 and
Propositions~23.5, 24.10 and~24.12.
\end{proof}

\begin{remark}\label{rk:cover}
Recall that a surjection from a projective module $P\to M$ is a projective cover if and only if
the kernel is contained in $\Rad(P)$.
\end{remark}

\begin{lemma}\label{le:cover}
Let $\AO$ satisfy \Cref{hyp:rad} and $M$ be a finitely
generated $\cO$-free $\AO$-module with projective cover $P\xrightarrow{f} M$. Then
the kernel of $P \xrightarrow{f} M$ is contained in
$P\cap \Rad(\hat P)$, and $\hat P
\to \hat M$ is a projective cover of the $\AK$-module $\hat
M=\K\otimes_\cO M$.
\end{lemma}
\begin{proof}
If $x\in\Ker(f)$ then by Remark~\ref{rk:cover}, $x \in
\Rad(P)$, which by \Cref{hyp:rad} is equal to
$\pi P+(P\cap\Rad(\hat P))$.
If there is such an $x$ which is not in $P\cap\Rad(\hat P)$, choose
one of the form $x=\pi^n z+y$ with $y\in\Rad(\hat P)$, $z\in P$, $z\not \in \Rad(P)$ and $n>0$.
Since $f(x)=0$ we have
\[ \pi^n f(z)=-f(y)\in \pi^n M\cap \Rad(\hat M)= f(\pi^n P\cap
  \Rad(\hat P)), \]
and we can write
$f(y)=\pi^n f(y')=f(\pi^n y')$ with $y' \in P\cap \Rad(\hat P)\subseteq \Rad(P)$. Thus
\[ \pi^n f(z+y')=f(\pi^nz)+f(\pi^n y')=f(\pi^nz+y)=f(x)=0 \]
and so $f(z+y')=0$. But $z+y'\not\in \Rad(P)$, contradicting Remark~\ref{rk:cover}.
\end{proof}

\begin{lemma}\label{le:min}
If $\AO$ satisfies \Cref{hyp:rad} and $M$ is a finitely generated $\cO$-free
$\AO$-module, and
\[ \dots \to P_1\to P_0\to M \to 0 \]
is a minimal projective resolution of $M$
over $\AO$, then after tensoring with $\K$, the sequence
\[ \dots \to \hat P_1\to \hat P_0 \to \hat M \to 0 \]
is a minimal projective resolution of $\hat M$ over $\AK$.
\end{lemma}
\begin{proof}
The tensored sequence is certainly a projective resolution. It follows
from \Cref{le:cover} that it is minimal.
\end{proof}

\begin{theorem}\label{th:hyp=>free}
 If $\AO$ satisfies \Cref{hyp:rad}, and
$M$ and $N$ are finitely generated and $\cO$-free, with $\hat N$ simple, then for all $t\ge 0$,
$\Ext^t_\AO(M,N)$ is $\cO$-free.
\end{theorem}
\begin{proof}
  Consider a minimal projective resolution of $M$
  \[ \cdots \to P_1 \to P_0 \to M \to 0. \]
  By \Cref{le:min}, this remains a minimal resolution after
  tensoring with $\K$. So if $\hat N$ is simple, the differential in
  the complex $\Hom_\AK(\hat P_*,\hat N)$ is zero. Since
  $\Hom_\AO(P_*,N)$ is contained in $\Hom_\cO(P_*,N)$, it is
  $\cO$-free, and $\Hom_\AK(\hat P_*,\hat
  N)=\K\otimes_\cO\Hom_\AO(P_*,N)$. So the differential on
  $\Hom_\AO(P_*,N)$ is also zero, and $\Ext^*_\AO(M,N)$ is $\cO$-free.
  \end{proof}

\begin{example}
The algebra $\AO=\cO[X]/(X^2)$ satisfies \Cref{hyp:rad}.
Let $M=\AO/(X)$ as an $\AO$-module, and let
$N$ be the ideal $\Rad(\AO)=(\pi,X)$ as an $\AO$-module.
The minimal resolution of $M$ is given by
\[ \cdots \to \AO \xrightarrow{X}\AO \xrightarrow{X}\AO\to M \to 0. \]
Ignoring the augmentation and taking homomorphisms to $N$, we get
\[ N \xrightarrow{X} N \xrightarrow{X} N \to \cdots \]
so $\Hom_A(M,N)\cong\cO$, and $\Ext^t_A(M,N)=(X)/(\pi X) \cong \F$ for all
$t>0$. This shows why $\hat N$
has to be simple in the proof of Theorem~\ref{th:hyp=>free}.
\end{example}

\section{Comparing $\Ext$ over $\cO$, $\F$ and $\K$}\label{S:ExtFree}

Let $\cO$, $\K$ and $\F$ be as in the introduction. For this section,
we let $\AO$ be an $\cO$-torsion free $\cO$-algebra. Notice that $\AO$ is not necessarily finitely generated. But we may continue to assume this if the reader wishes to.  Let $\AF=\F\otimes_\cO A$ and $\AK=\K\otimes_\cO A$. Then $\AK$ is
flat as an $\AO$-module.

\begin{theorem}\label{th:KOF}
  Let $M$ and $N$ be $\cO$-torsion free $A$-modules.
\begin{enumerate}[\rm (i)]
\item  Suppose that for $t\ge 0$,  $\Ext^t_\AO(M,N)$ is
  $\cO$-torsion free. Then
for $t\ge 0$, we have an isomorphism
\[ \F\otimes_\cO\Ext^t_\AO(M,N) \cong \Ext^t_\AF(\bar M,\bar N). \]
\item  If $M$ is finitely presented, then for $t\ge 0$, we have an isomorphism
  \[ \K \otimes_\cO \Ext^t_\AO(M,N) \cong \Ext^t_\AK(\hat M,\hat N). \]
\end{enumerate}
\end{theorem}
\begin{proof}
(i) Tensoring the short exact sequence
\[ 0 \to \AO \xrightarrow{\pi} \AO\to \AF \to 0. \]
with $M$, we get an exact sequence
\[ \cdots \to \Tor^\AO_1(\AO,M) \to \Tor^\AO_1(\AF,M)
  \to M \xrightarrow{\pi} M \to \bar M \to 0. \]
We have $\Tor_s^\AO(\AO,M)=0$ for $s\ge 1$, so $\Tor^\AO_s(\AF,M)=0$
for $s\ge 2$.
Since $M$ is $\cO$-torsion free and multiplication with $\pi$ is injective, it also follows that $\Tor^\AO_1(\AF,M)=0$. By Cartan and Eilenberg~\cite[VI, Proposition 4.1.3]{Cartan/Eilenberg:1956a}, we get an isomorphism
\begin{equation}\label{eq:Ext}
  \Ext^t_\AF(\bar M,\bar N) \xrightarrow{\cong}
  \Ext^t_\AO(M,\bar N).
\end{equation}
Next, since $N$ is $\cO$-torsion free, we have a short exact sequence
\[ 0 \to N \xrightarrow{\pi} N \to \bar N \to 0. \]
This gives us a long exact sequence
\[ \cdots \to \Ext^t_\AO(M,N) \xrightarrow{\pi} \Ext^t_\AO(M,N) \to
  \Ext^t_\AO(M,\bar N)\to  \cdots \]
By the hypothesis, $\Ext^t_\AO(M,N)$ is $\cO$-torsion free, so multiplication by $\pi$ is
injective. The sequence therefore breaks up into short exact sequences,
which shows that
\begin{equation}\label{eq:Ext2}
  \Ext^t_\AO(M,\bar N)\cong\F \otimes_\cO \Ext^t_\AO(M,N).
\end{equation}
Combining~\eqref{eq:Ext} with~\eqref{eq:Ext2} proves part (i).

(ii) Consider the exact sequence \[\cdots\to P_1\to P_0\to M\to 0\] with $P_t$'s projective $A$-modules. Since $\K$ is flat over $\cO$, tensoring with $\K$ over $\cO$, we have the exact sequence $\cdots\to \hat P_1\to \hat P_0\to \hat M\to 0$ where $\hat P_t$'s are projectives. Since $\K\otimes_\cO \Hom_A(P_t,N)\cong \Hom_{\AK}(\hat P_t,\hat N)$, we get the isomorphism $\K \otimes_\cO \Ext^t_\AO(M,N) \cong \Ext^t_\AK(\hat M,\hat N)$.
\end{proof}

\section{$\cO$-freeness of $J^n/J^{n+1}$}
  \label{se:Ofree}

In this section, we investigate \Cref{hyp:basic}. We begin
with a preliminary lemma.

\begin{lemma}\label{le:Ofree}
Let $X$ be a finitely generated free $\cO$-module and let $Y$ be a
submodule. Let $\hat X=\K\otimes_\cO X$, $\hat Y=\K\otimes_\cO Y$,
$\bar X=\F\otimes_\cO X$, and $\bar Y$ be the image of
$\F\otimes_\cO Y\to \F\otimes_\cO X$. Then we have
\[ \dim_\F\bar Y \le \dim_\K \hat Y \]
with equality if and only if $X/Y$ is $\cO$-free.
\end{lemma}
\begin{proof}
We have $\hat X/\hat Y\cong\K\otimes_\cO(X/Y)$, and
\[ \xymatrix@=5mm{0 \ar[r]& \Tor_1^\cO(\F,X/Y) \ar[r]&
\F\otimes_\cO Y \ar[rr]\ar[dr]&&
\F\otimes_\cO X \ar[r]&\F\otimes_\cO (X/Y) \ar[r]& 0.\\
&&&\bar Y\ar[ur]\ar[dr]\\
&& 0\ar[ur]&&0.} \]
Since $Y$ is $\cO$-free, we have
\begin{align*}
\dim_\F \bar Y&=\dim_\F\F\otimes_\cO Y - \dim_\F\Tor_1^\cO(\F,X/Y)\\
&\le \dim_\F\F\otimes_\cO Y=\rank_\cO Y=\dim_\K\hat Y
\end{align*}
with equality if and only if $\Tor_1^\cO(\F,X/Y)=0$. Since $X/Y$ is
finitely generated, this is equivalent to the condition that
it is $\cO$-free.
\end{proof}

The following lemma is useful in verifying \Cref{hyp:freeJ}.

\begin{lemma}\label{le:Ofree-equiv} Let $J$ be a nilpotent ideal of $A$ and suppose that $A$ is $\cO$-free of finite rank. Then the following are
equivalent:
\begin{enumerate}[\rm (i)]
\item For all $n\ge 1$, $A/J^n$ is $\cO$-free.
\item For all $n\ge 0$, $J^n/J^{n+1}$ is $\cO$-free.
\item
We  can find an $\cO$-basis $\cB$
of $\AO$ and a descending chain of subsets
$\cB\supseteq\cB_1\supseteq\cB_2\supseteq\cdots$ such that each
$\cB_i$ is an $\cO$-basis for $J^i$.
\end{enumerate} Furthermore, if $J=\JA=\AO\cap\Rad(\AK)$, then the following are equivalent to the statements (i)--(iii) above.
\begin{enumerate}[\rm (i)]\addtocounter{enumi}{3}
\item For all $n\ge 1$, we have
$\dim_\F\Rad^n(\AF)\ge\dim_\K\Rad^n(\AK)$.
\item For all $n\ge 1$, we have
$\dim_\F\Rad^n(\AF)=\dim_\K\Rad^n(\AK)$.
\end{enumerate}
\end{lemma}
\begin{proof}
(i) $\Leftrightarrow$ (ii): If $\AO/J^{n+1}$ is $\cO$-free then so is
the submodule $J^n/J^{n+1}$. Conversely, $\AO/J$ is $\cO$-free with
basis the vertex idempotents. So if
each $J^n/J^{n+1}$ is $\cO$-free then $\AO/J^n$ has a finite
filtration with $\cO$-free subquotients, so it is $\cO$-free.

(ii) $\Rightarrow$ (iii):
Since $\AO/J$ is $\cO$-free,
we can choose an $\cO$-basis. Then $\AO\to\AO/J$ is $\cO$-split, so we
let $\cB\setminus\cB_1$ be the lift of this basis using the
splitting. Suppose by induction on $n\ge 1$ that we have chosen a free
$\cO$-basis for $\AO/J^n$ and used a splitting of $\AO\to\AO/J^n$ to
lift to $\cB\setminus\cB_n$. Then $J^n$ is $\cO$-free, and
$J^n/J^{n+1}$ is $\cO$-free, so $J^n\to J^n/J^{n+1}$ is $\cO$-split,
and we apply a splitting to a free basis of $J^n/J^{n+1}$ to get
$\cB_n\setminus \cB_{n+1}$. This gives us a set
$\cB\setminus\cB_{n+1}$ whose image in $\AO/J^{n+1}$ is a free
$\cO$-basis. When
$n$ is sufficiently large, $J^n=0$ and we are done.

(iii) $\Rightarrow$ (ii):
Given such subsets
$\cB\supseteq\cB_1\supseteq\cB_2\supseteq\cdots$, the image of
$\cB_n\setminus\cB_{n+1}$ in $J^n/J^{n+1}$ is a free $\cO$-basis.

Suppose further that $J=\JA$.

(iv) $\Leftrightarrow$ (v) $\Leftrightarrow$ (i):
We have $\Rad(\AK)=\hat J$ and $\Rad(\AF)=\bar J$.
So applying \Cref{le:Ofree} with $X=A$ and $Y=J^n$,
we have $\dim_\F\Rad^n(\AF)\le\dim_\K\Rad^n(\AK)$, with equality if and
only if $\AO/J^n$ is $\cO$-free.
\end{proof}

\begin{theorem}\label{th:freeJ}
If \Cref{hyp:freeJ} holds then the radical layer multiplicities of $\hat P_i$ and $\bar P_i$ are equal, i.e., for all $1\le i,j\le m$ and $n\ge  0$,
\[ [\Rad^n(\hat P_i)/\Rad^{n+1}(\hat P_i):\hat S_j]=
[\Rad^n(\bar P_i)/\Rad^{n+1}(\bar P_i):\bar S_j]. \]
\end{theorem}
\begin{proof} Let $J=\JA$. Since $A$ satisfies \Cref{hyp:basic}, $J$ is nilpotent by \Cref{th:basic}(ii). By \Cref{le:Ofree-equiv}, we have an $\cO$-basis $\Omega$ for the $\AO$-module $J^n/J^{n+1}$. Let $\hat e_i,\hat e_j\in \AK$ and $\bar e_i,\bar e_j\in \AK$ be the corresponding vertex idempotents. Furthermore, by assumption, we have
\begin{align*}
\hat s_{ij}:=&[\Rad^n(\hat P_i)/\Rad^{n+1}(\hat P_i):\hat S_j]=\dim_\K\hat e_j(\K\otimes_\cO(J^n/J^{n+1}))\hat e_i=\Sp_\K\{\hat e_j\hat b\hat e_i:b\in\Omega\},\\
\bar s_{ij}:=&[\Rad^n(\bar P_i)/\Rad^{n+1}(\bar P_i):\bar S_j]=\dim_\F\bar e_j(\F\otimes_\cO(J^n/J^{n+1}))\bar e_i=\Sp_\F\{\bar e_j\bar b\bar e_i:b\in\Omega\}.
\end{align*} Since $\hat e_j\hat b\hat e_i\in A$, we get $\hat s_{ij}\ge  \bar s_{ij}$. Summing over all $i,j$, we have \[\sum_{i,j}\hat s_{ij}=\dim_\K (\K\otimes_\cO(J^n/J^{n+1}))=\dim_\F (\F\otimes_\cO(J^n/J^{n+1}))=\sum_{i,j}\bar s_{ij}\] and get $\hat s_{ij}= \bar s_{ij}$.
\end{proof}

\begin{example}
Let $Q$ be the quiver
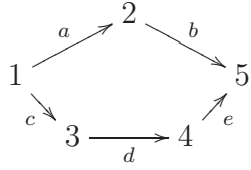
\[ \xymatrix@=3mm{&&2\ar[drr]^b&&\\
    1\ar[urr]^a\ar[dr]_c&&&&5\\
    &3\ar[rr]_d&&4\ar[ur]_e&} \]
and $\AO=\cO Q/I$ with $I=(edc-\pi ba)$. Let $J=\JA$. The algebra $A$ satisfies
\Cref{hyp:basic}, but not \Cref{hyp:freeJ},
because $ba$ is in $J^2$ but not in $J^3$ while $\pi ba$ is in $J^3$.
So as an $\cO$-module, $J^2/J^3\cong \F\oplus \cO\oplus \cO$, spanned
by $ba$, $dc$ and $ed$. Thus $A/J^3\cong\F\oplus \cO^{\oplus 12}$ is not $\cO$-free.

The projective covers of the
module $1$ over $\AK$ and $\AF$ are given by the following diagrams respectively:
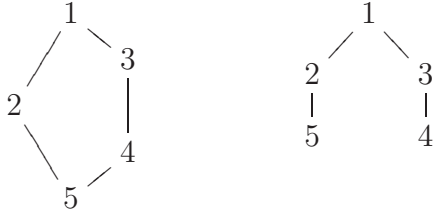
\[\xymatrix@R=1mm@C=3mm{&1\ar@{-}[ddl]\ar@{-}[dr]&\\
    &&3\ar@{-}[dd]\\2\ar@{-}[ddr]&&\\&&4\ar@{-}[dl]\\&5&} \hspace{2cm}
\xymatrix@=3mm{&1\ar@{-}[dl]\ar@{-}[dr]&\\2\ar@{-}[d]&&3\ar@{-}[d]\\5&&4} \] In particular, the radical lengths are different.
\end{example}



\section{The descent algebras}\label{S:Des}

In this section, we review some elementary properties for (finite) Coxeter groups and their descent algebras for the discussions in \Cref{S:Idem,S: proj,S: indep}. For the general knowledge about Coxeter groups, we refer the reader to \cite{Geck/Pfeiffer:2000a,Humphreys:1990a}. We will investigate the algebra over a general field $k$ of characteristic $q$ (see \Cref{not:A}).

Irreducible finite Coxeter groups can be classified in terms of the Coxeter-Dynkin diagrams. Let $(W,S)$ be a (finite irreducible) Coxeter system and we call $W$ a Coxeter group. A Coxeter element of $W$ is a product of all $s\in S$ in any given order. It is known that any two Coxeter elements of $W$ are conjugate in $W$ by a cyclic shift. As such, by abuse, we may speak of the Coxeter element of $W$. For each subset $J\subseteq S$, we write $W_J$ for the parabolic subgroup of $W$ generated by $J$. The pair $(W_J,J)$ is again a Coxeter system. We write $c_J$ for the Coxeter element of $W_J$.

For two subsets $J,K$ of $S$, we write $J\subseteq_W K$ (respectively, $J=_WK$) if there exists $w\in W$ such that ${}^wJ\subseteq K$ (respectively, ${}^wJ=K$). In the case when $J=_WK$, we say that $J$ and $K$ are conjugate in $W$. In the case when $J\subseteq_W K$ but $J\neq_WK$, we write $J\subsetneq_WK$. The following theorem is well-known.

\begin{theorem} Let $(W,S)$ be a Coxeter system and $J,K$ be subsets of $S$. The following statements are equivalent.
\begin{enumerate}[\rm (i)]
\item $W_J$ and $W_K$ are conjugate in $W$,
\item $J=_WK$,
\item $c_J$ and $c_K$ are conjugate in $W$.
\end{enumerate}
\end{theorem}

Let $\ell:W\to \Z_{\ge  0}$ be the length function for $(W,S)$. For each $J\subseteq S$, let $X_J$ be the distinguished set of left coset representatives of $W_J$ in $W$ consisting of minimal length representatives, that is, \[X_J=\{w\in W: \text{$\ell(ws)>\ell(w)$ for all $s\in J$}\}.\]  The set $X_J^{-1}=\{w^{-1}:w\in X_J\}$ is thus the distinguished set of minimal length right coset representatives of $W_J$ in $W$. For another subset $K\subseteq S$, $X_{JK}:=X_J^{-1}\cap X_K$ is therefore a set of double coset representatives of $(W_J,W_K)$ in $W$. We let \[x_J=\sum_{w\in X_J}w\in \Z W.\]  Solomon \cite{Solomon:1976a} showed that, over $\Z$, for $J,K\subseteq S$, \[x_Jx_K=\sum_{L\subseteq S}a_{JKL}x_L\] where $a_{JKL}$ is the number of elements $w\in X_{JK}$ such that $L=w^{-1}Jw\cap K$. Thus the $\Z$-span of the set $\{x_J:J\subseteq S\}$ is an $\Z$-subalgebra of $\Z W$ and it is called the Solomon descent algebra of $W$ (over $\Z$). We denote it $\Des_\Z$. Furthermore, $\Des_\Z$ is $\Z$-free of rank $2^{|S|}$. Since the structure constants $a_{JKL}$ are integers, for any integral domain $R$, considering $a_{JKL}\cdot 1_R$, we obtain the descent algebra $\Des_R:=R\otimes_\Z\Des_\Z$ which is considered as a subalgebra of $RW$. The set $\{x_J:J\subseteq S\}$ forms an $\k$-basis for $\Des_\k$. Let $Y_J$ and $Y^\circ_J$ be the subspaces of $\Des_\k$ spanned by $\{x_K:K\subseteq_WJ\}$ and $\{x_K:K\subsetneq_W J\}$ respectively. We have the following.

\begin{lemma}\label{L: 2-sided ideal Y} Let $J\subseteq S$. The subspaces $Y_J$ and $Y^\circ_J$ are two-sided ideals of $\Des_\k$. Furthermore, $Y^\circ_J\subsetneq Y_J$.
\end{lemma}

In \cite{Solomon:1976a}, Solomon also determined the radical $\Rad(\Des_\Q)$. Together with the work by Atkinson--van Willigenburg \cite{Atkinson/vanWilligenburg:1997a} (for the type $\mathbb{A}$ case) and  Atkinson--Pfeiffer--van Willigenburg \cite{Atkinson/Pfeiffer/vanWilligenburg:2002a} (for the general case), the radical of $\Des_\k$ is known. We now describe their results.

For each $J\subseteq S$, let $\varphi_J$ be the permutation character of the induced module from the trivial module for $W_J$ to $W$. Notice that $\varphi_J$ takes integer values on the elements of $W$. Let $\ccl_\Z$ be the $\Z$-span of the set $\{\varphi_J:J\subseteq S\}$. We have the well-known Mackey formula \[\varphi_J\varphi_K=\sum_{L\subseteq S}a_{JKL}\varphi_L\] where $a_{JKL}$'s are precisely the integers we have obtained earlier. Therefore, we obtain a homomorphism of $\Z$-algebras \[\theta:\Des_\Z\to \ccl_\Z\] given by $\theta(x_J)=\varphi_J$. Let $\ccl_{\k}$ be the $\k$-span of the set $\{\varphi^{\k}_J:J\subseteq S\}$ where $\varphi^{\k}_J(x)=\varphi_J(x)\cdot 1_\k$ for each $x\in W$. As such, tensoring with $\k$, the map $\theta$ reduces to an $\k$-algebra homomorphism $\theta_\k:\Des_\k\to \ccl_{\k}$ where $\theta_\k(x_J)=\varphi^{\k}_J$.

\begin{theorem}[{\cite[Theorem 3]{Solomon:1976a}, \cite[Theorem 3]{Atkinson/Pfeiffer/vanWilligenburg:2002a}}]\label{T: Sol Radical} The radical $\Rad(\Des_\k)$ is spanned by elements $x_J - x_K$ such that $J =_W K$ and together with elements $x_L$ such that $q\mid [\N_W(W_L):W_L]$. Furthermore, $\Rad(\Des_\k)=\Ker\theta_\k$. In particular, $\Des_\k$ is a basic algebra.
\end{theorem}

Suppose now that $n>0$ and let $\I(n)$ be the set consisting of pairs $(\un{J},\un{K})$ of sequences $\un{J}=(J_1,\ldots,J_n)$ and $\un{K}=(K_1,\ldots,K_n)$ of subsets $S$ such that $J_i=_W K_i$ for each $i\in [1,n]$. In this case, we write $\un{J}=_W\un{K}$. For $(\un{J},\un{K})\in \I(n)$, denote \[x_{\un{J},\un{K}}=\prod_{i=1}^n(x_{J_i}-x_{K_i})\in \Rad^n(\Des_\k).\]  For example, if $n=1$ and $J=_WK$, we have $x_{(J),(K)}=x_J-x_K\in\Rad(\Des_\k)$. By \Cref{T: Sol Radical}, if $q\nmid |W|$, we have \[\Rad^n(\Des_\k)=\Sp_\k\{x_{\un{J},\un{K}}:(\un{J},\un{K})\in \I(n)\}.\]

Therefore, we have the following lemma.

\begin{lemma}\label{L: basis for J^r} Let $(W,S)$ be a Coxeter group, $n$ be a positive integer and suppose that $q\nmid |W|$. There is a subset $T\subseteq \I(n)$ such that $\{x_{\un{J},\un{K}}:(\un{J},\un{K})\in T\}$ forms a basis for $\Rad^n(\Des_\k)$.
\end{lemma}

Consider the action of $W$ on the subsets of $S$. The equivalence classes are called the Coxeter classes of $W$. Let $\R$ be a complete set of representatives of the Coxeter classes and we also call the elements in $\R$ Coxeter classes. We fix a total order $\Rleq$ for $\R$ such that, if $J\subseteq_W K$, we have $J\Rleq K$. For $J,K\subseteq S$, we define the mark and the parabolic table of marks of $W$ as $\beta_{JK}=|\Fix_{W_K}(W/W_J)|$ and $M=(\beta_{JK})_{J,K\in\R}$ respectively.

\begin{lemma}
[{\cite[Proposition 2.4.4]{Geck/Pfeiffer:2000a}
and \cite[Lemma 1]{Atkinson/Pfeiffer/vanWilligenburg:2002a}}]\label{L:beta}
Let $(W,S)$ be a Coxeter system and $J,K$ be subsets of $S$. We have \[\beta_{JK}=\varphi_J(c_K)=a_{JKK}=[\N_W(W_J):W_J]\cdot |\{W^\omega_J:\omega\in W,\ W_K\subseteq W_J^\omega\}|.\] Furthermore, we have the following:
\begin{enumerate}[\rm (i)]
\item $\beta_{JJ}=[\N_W(W_J):W_J]$ and $\beta_{JJ}\mid \beta_{JK}$.
\item If $\beta_{JK}\neq 0$, then $K\subseteq_W J$. In particular, the matrix $M$ is lower triangular.
\end{enumerate}
\end{lemma}


Let $\R_q$ be the subset of $\R$ consisting of $J\in\R$ such that $q\nmid \beta_{JJ}=[\N_W(W_J):W_J]$ and $M_\k$ be the entry-wise modulo $q$ of the principal submatrix of $M$ labelled by $\R_q$, that is, for $J,K\in\R_q$, \[(M_\k)_{J,K}=\beta_{JK}\cdot 1_k=\varphi^{\k}_J(c_K).\] Notice that $M_\k$ is also lower triangular and $\R_0=\R$. We call the elements in $\R_q$ the $q$-Coxeter classes. For example, in the $\mathbb{A}_{n-1}$ case, $\R$ and $\R_q$ are labelled by partitions and $q$-regular partitions of $n$ respectively. 


\begin{lemma}\label{L: qCox max} Let $J\in\R_q$ and $I\in \R$ such that $\varphi^{\k}_K(c_I)=\varphi^{\k}_K(c_J)$ for all $K\subseteq S$. Then $I\subseteq_W J$.
\end{lemma}
\begin{proof} Since $\varphi^{\k}_J(c_I)=\varphi^{\k}_J(c_J)\neq 0$, by \Cref{L:beta}, we have  $I\subseteq_W J$.
\end{proof}

Furthermore, we have the following lemma.

\begin{lemma}[{\cite[Lemmas 2, 3]{Atkinson/Pfeiffer/vanWilligenburg:2002a}}]\label{L: matrix M_\k} The matrix $M_\k$ is lower triangular of rank $|\R_q|$. Furthermore, the complete set of non-isomorphic simple $\Des_\k$-modules is labelled by $\R_q$ and defined by the columns of $M_\k$ in the sense that, for each $J\in \R_q$, the simple $\Des_\k$-module $S_{J,k}$ is one-dimensional such that $x\in \Des_\k$ acts via the multiplication by $\theta_\k(x)(c_J)\in \k$.
\end{lemma}


For each $J\subseteq S$, let $\Char_{J,\k}$ denote the characteristic function on the Coxeter class labelled by $J$ over $\k$, that is, for each $K\in \R_q$, \[\Char_{J,\k}(c_K)=\left \{\begin{array}{ll} 1&\text{if $K=J$,}\\ 0&\text{otherwise.}\end{array}\right .\] By \Cref{L: matrix M_\k}, $\{\Char_{J,\k}:J\in\R_q\}$ forms a basis and complete set of primitive orthogonal idempotents for $\im \theta_\k$.



We now investigate the simple modules for the descent algebras further. Let $J\in\R$ and $S_{J,\Z}$ be the free $\Z$-module of rank 1 and let $x_K\in\Des_\Z$ act via the multiplication by $\theta(x_K)(c_J)=\varphi_K(c_J)$. Notice that $\k\otimes_\Z S_{J,\Z}\cong \k\otimes_\Z S_{J',\Z}$ as $\Des_\k$-modules if and only if $\varphi^{\k}_K(c_J)=\varphi^{\k}_K(c_{J'})$ for all $K\subseteq S$. Furthermore, if $J\in\R_q$, then $S_{J,\k}\cong \k\otimes_\Z S_{J,\Z}$. 

Let $A$ be a finite-dimensional $\k$-algebra. Recall that we have a contravariant exact functor \[\D(-)=\Hom_\k(-,\k):\text{\textsf{mod}-$A$}\to \text{$A$-\textsf{mod}}\] from the category of right $A$-modules to the category of left $A$-modules. More precisely, for a right $A$-module $M$ and $\xi\in\Hom_\k(M,\k)$, we have $(a\xi)(m)=\xi(ma)$ for all $a\in A$ and $m\in M$. Furthermore, $M$ is projective if and only if $\D(M)$ is injective.

Let $P_{J,\k}$ and $I_{J,\k}$ be the projective cover and injective hull of the simple module $S_{J,\k}$ respectively. We have the following proposition.

\begin{proposition}\label{P: proj cover} Let $J\in \R_q$ and $e\in \Des_\k$ be a primitive idempotent such that $\theta_\k(e)=\Char_{J,\k}$. Then $P_{J,\k}\cong \Des_\k e$ and $I_{J,\k}\cong \D(e\Des_\k)$.
\end{proposition}
\begin{proof} Let $P=\Des_ke$ and $K\subseteq S$. Since \[\theta_\k((x_K-\varphi_K^{\k}(c_J))e)=(\varphi^{\k}_K-\varphi^{\k}_K(c_J))\Char_{J,\k}=0,\] we have $(x_K-\varphi^{\k}_K(c_J))e\in\Ker\theta_\k=\Rad(\Des_k)$. As such, $x_K$ acts by multiplication by $\varphi_K^{\k}(c_J)$ on $Q:=P/\Rad(P)$ and hence $Q$ is isomorphic to $S_{J,\k}$. So $P\cong P_{J,\k}$.

Now consider the right $\Des_k$-module $P':=e\Des_k$. Similar as in the previous paragraph, $P'$ is the projective cover of the right simple $A$-module $T:=S'_{J,\k}$ where $T$ is one-dimensional spanned by $\{v\}$ and $v\cdot x_K=\varphi^{\k}_K(c_J)v$ for each $K\subseteq S$. Therefore $\D(P')$ is the injective hull of $\D(T)=\Hom_\k(T,\k)$. We only need to check that $\D(T)\cong S_{J,\k}$. Let $\{\xi\}$ be the dual basis of $\D(T)$ with respect to $\{v\}$. We have, for any $K\subseteq S$,  \[(x_K\xi)(v)=\xi(v\cdot x_K)=\xi(\varphi^{\k}_K(c_J)v)=\varphi^{\k}_K(c_J).\] Therefore, $x_K\xi=\varphi^{\k}_K(c_J)\xi$ and $\D(T)\cong S_{J,\k}$ as desired.
\end{proof}

\section{Idempotents for the descent algebras}\label{S:Idem}

In this section, we generalise the construction of the idempotents for the descent algebras in the type $\mathbb{A}$ case (see \cite{Lim:2023a}) to arbitrary Coxeter groups. Since the construction is similar, we omit some proofs and computational details but keep the notations as close as possible as in \cite{Lim:2023a}.

As before, throughout this section, $\k$ is a field of characteristic $q$ ($q=0$ or $q>0$). For each $J\in \R_q$, let \[\gamma_J=\beta_{JJ}\cdot 1_\k=\varphi^{\k}_J(c_J).\] By definition, $\gamma_J\neq 0$. Recall the total order $\Rleq$ we have fixed earlier for the set $\R$ (and hence on the subset $\R_q$). Let $M_\k^{-1}=(b_{JK})_{J,K\in \R_q}$ be the inverse matrix of $M_\k=(\varphi^{\k}_J(c_K))_{J,K\in\R_q}$. For each $J\in \R_q$, define
\begin{align}\label{Eq:fJ}
f_J=\sum_{K\in\R_q} b_{JK}x_K\in \Des_\k.
\end{align} Furthermore, recall  the two-sided ideals $Y^\circ_J$ and $Y_J$ of $\Des_\k$ as in  \Cref{L: 2-sided ideal Y} .

The proof of the following lemma is similar to that of \cite[Lemma 3.2]{Lim:2023a} and we shall leave it to the readers.

\begin{lemma}\label{L: f*} Let $(W,S)$ be a Coxeter system and let $J\in \R_q$.
\begin{enumerate}[\rm (i)]
\item We have $\theta_\k(f_J)=\Char_{J,\k}$.
\item For any $\R_q\ni K\not\subseteq_W J$, we have $b_{JK}=0$. In particular, \[\textstyle f_J=\frac{1}{\gamma_J}x_J+\sum_{\R_q\ni K\subsetneq_W J} b_{JK}x_K\in Y_J.\]
\item For any positive integer $r$, we have $(f_J)^r=\frac{1}{\gamma_J}x_J+\epsilon_{J,r}$ for some $\epsilon_{J,r}\in Y^\circ_J$.
\item We have $\sum_{J\in \R_q}f_J=1$.
\end{enumerate}
\end{lemma}

As $\theta_\k(f_J)=\Char_{J,\k}$ and $\Ker\theta_\k=\Rad(\Des_\k)$, we identify $\Des_\k/\Rad(\Des_\k)$ with $\ccl_{\k}$ via the $\k$-algebra epimorphism $\theta_\k$ and can lift and orthogonalise the set \[\{f_J+\Rad(\Des_\k):J\in\R_q\}\] to obtain a complete set of orthogonal primitive idempotents for $\Des_\k$. Let $m=|\R_q|$. We denote \[\R_q=\{J_i:i \in [1,m]\},\] where $J_i<J_{i+1}$ for each $i\in [1,m-1]$ and define $f_i=f_{J_i}$, $f_{\ge  i}=\sum_{j\ge  i} f_j$, $\Char_i=\Char_{J_i,\k}$, $\Char_{\ge  i}=\sum_{j\ge  i}\Char_i$, $\gamma_i=\gamma_{J_i}$,
$Y_{\le i}=\sum_{k\in [1,i]} Y_{J_k}$ and $Y^\circ_{\le i}=Y^\circ_{J_i}+\sum_{k\in [1,i-1]} Y_{J_k}$ one for each $i\in [1,m]$. Notice that $J_m=S$.

By \Cref{L: 2-sided ideal Y,L:beta}, we have the following result.

\begin{lemma}\label{L: Y0} Let $i\in [1,m]$. We have
\begin{enumerate}[\rm (i)]
\item $Y^\circ_{\le i}$ and $Y_{\le i}$ are two-sided ideals of $\Des_\k$,
\item $x_{J_i}^r=\gamma_i^{r-1}x_{J_i}(\mod Y^\circ_{\le i})$ for any positive integers $r$, and
\item $Y^\circ_{\le 1}\subsetneq Y_{\le 1}\subsetneq Y^\circ_{\le 2}\subsetneq Y_{\le 2}\cdots \subsetneq Y^\circ_{\le m}\subsetneq Y_{\le m}=\Des_\k$.
\end{enumerate}
\end{lemma}

Since we are dealing with the characteristic zero case ($q=0$) as well, we need to replace the idempotent-lifting procedure in \cite[Definition 3.5]{Lim:2023a} with the $(3a^2-2a^3)$-construction (see the proof of \cite[Theorem
  1.7.3]{Benson:1991a}) and provide proofs for both \Cref{L: f'} and \Cref{T: idem}. For this, we need to introduce some extra notation as follows.

\begin{proposition}\label{P: lift idem} Let $N$ be a nilpotent ideal of a finite-dimensional $\k$-algebra $A$ and $f\in A$ such that $f+N\in A/N$ is an idempotent. Let $a_1=f$ and, inductively, for $i\ge  1$, let \[a_{i+1}=3a_i^2-2a_i^3.\] Then there exists an integer $r$ such that $a_r$ is an idempotent of $A$ and thus $a_{r'}=a_r$ for any $r'\ge  r$. We denote $a_\infty=a_r$. Furthermore, $a_i$ commutes with $a_1$ for all $i\ge  1$ and \[(a_1a_\infty)_\infty=a_\infty.\]
\end{proposition}
\begin{proof} The first assertion is well-known. The fact that $a_i$ commutes with $a_1$ can be proved easily using induction on $i$. Notice that $(a_1a_\infty)+N$ is again an idempotent. To show that $(a_1a_\infty)_\infty=a_\infty$, let $b_1=a_1a_\infty$ and, inductively, for all $i\ge 1$, let $b_{i+1}=3b_i^2-2b_i^3$. We check
  that $b_i=a_ia_\infty$ for all $i\ge  1$ using induction. Since $a_\infty$ is an idempotent, we
  have
\begin{align*}
b_{i+1}&=3b_i^2-2b_i^3=3(a_ia_\infty)^2-2(a_ia_\infty)^3=3a_i^2a_\infty^2-2a_i^3a_\infty^3=3a_i^2a_\infty-2a_i^3a_\infty=a_{i+1}a_\infty.
\end{align*} As such $b_\infty=a_\infty^2=a_\infty$.
\end{proof}

In the case of the descent algebra, notice that we have $\theta_\k(f_{i-1}'f_{\ge  i}f_{i-1}')=\Char_{\ge  i}$. This allows us to define the following.

\begin{definition}\label{D: idem} In $\Des_\k$, let $f'_1=1$ and, inductively, for $i\in [2,m]$, we define $f_i'=(f_{i-1}'f_{\ge  i}f_{i-1}')_\infty$ and $f_{m+1}'=0$. For each $i\in [1,m]$, define \[e_{J_i}=f_i'-f_{i+1}'.\]
\end{definition}

The following lemma is similar to \cite[Lemma 3.6]{Lim:2023a}.

\begin{lemma}\label{L: f'}\
\begin{enumerate}[\rm (i)]
\item For each $i\in [1,m]$, the element $f_i'$ is an idempotent such that $\theta_\k(f_i')=\Char_{\ge  i}$ and for $i>1$, we have $f_i'\in f_{i-1}'\Des_\k f_{i-1}'$.
\item For $i\in [2,m+1]$, \[\textstyle f_i'=1-\left (\frac{1}{\gamma_{i-1}}x_{J_{i-1}}+\epsilon_{i-1}\right )\] for some $\epsilon_{i-1}\in Y^\circ_{\le i-1}$.
\end{enumerate}
\end{lemma}
\begin{proof} We first prove part (i) by using induction on $i$. It is clearly true when $i=1$. Suppose that $i>1$ and assume that the statement holds true for $i-1$. Let $a_1=f_{i-1}'f_{\ge  i}f_{i-1}'$ and $a_{r+1}=3a_r^2-2a_r^3$ for $r\ge  1$. Let $A=f_{i-1}'\Des_\k f_{i-1}'$. Notice that $a_1\in A$ and, by \Cref{L: f*}, \[\theta_\k(a_1)=\theta_\k(f_{i-1}'f_{\ge  i}f_{i-1}')=\Char_{\ge  i-1}\Char_{\ge  i}\Char_{\ge  i-1}=\Char_{\ge  i}.\] Suppose that $\theta_\k(a_r)=\Char_{\ge  i}$ and $a_r\in A$ for some positive integer $r$. We have \[\theta_\k(a_{r+1})=\theta_\k(3a_r^2-2a_r^3)=3\ \Char_{\ge  i}^2-2\ \Char_{\ge  i}^3=\Char_{\ge  i}.\] As such, $f_i'=a_\infty=(f_{i-1}'f_{\ge  i}f_{i-1}')_\infty$ is an idempotent of $\Des_\k$ such that $\theta_\k(f_i')=\Char_{\ge  i}$. Furthermore, as $f_{i-1}'$ is an idempotent and $a_r\in A$, \[a_{r+1}=3a_r^2-2a_r^3=f_{i-1}'(3a_r^2-2a_r^3)f_{i-1}'\in A.\] Therefore, $f_i'=a_\infty\in A$.

We again argue by induction on $i$ to prove part (ii). Suppose that $i=2$. Let $a_1=f_{\ge  2}=1-f_1\equiv 1-\frac{1}{\gamma_1}x_{J_1}(\mod Y^\circ_{\le 1})$ and $a_{r+1}=3a_r^2-2a_r^3$ as before. Suppose that $a_r\equiv 1-\frac{1}{\gamma_1}x_{J_1}(\mod Y^\circ_{\le 1})$ for some $r\ge  1$.  By \Cref{L: Y0}, we have
\begin{align*}
a_{r+1}&\equiv {\textstyle 3(1-\frac{1}{\gamma_1}x_{J_1})^2-2(1-\frac{1}{\gamma_1}x_{J_1})^3}(\mod Y^\circ_{\le 1})\\
&= {\textstyle3(1-\frac{2}{\gamma_1}x_{J_1}+\frac{1}{\gamma_1}x_{J_1})-2(1-\frac{3}{\gamma_1}x_{J_1}+\frac{3}{\gamma_1}x_{J_1}-\frac{1}{\gamma_1}x_{J_1})}\\
&={\textstyle1-\frac{1}{\gamma_1}x_{J_1}}.
\end{align*} As such, $f_2'=a_\infty$ has the desired form as in the statement. Suppose now that $f_i'=1-(\frac{1}{\gamma_{i-1}}x_{J_{i-1}}+\epsilon_{i-1})$ for some $\epsilon_{i-1}\in Y^\circ_{\le i-1}$ and $\epsilon_{i-1}'=\frac{1}{\gamma_{i-1}}x_{J_{i-1}}+\epsilon_{i-1}\in Y_{\le i-1}$. Again, let $a_1=f_i'f_{\ge  {i+1}}f_i'$ and $a_{r+1}=3a_r^2-2a_r^3$ for $r\ge  1$. By Lemmas \ref{L: f*} and \ref{L: Y0}, let $f_i=\frac{1}{\gamma_i}x_{J_i}+z_i$ where $z_i\in Y^\circ_{\le i}$, we have
\begin{align*}
a_1&=(1-\epsilon_{i-1}')(1-f_1-\cdots-f_i)(1-\epsilon_{i-1}')\\
&={\textstyle(1-\epsilon_{i-1}')(1-f_1-\cdots-f_{i-1}-(\frac{1}{\gamma_i}x_{J_i}+z_i))(1-\epsilon_{i-1}')}\\
&\equiv {\textstyle1-\frac{1}{\gamma_i}x_{J_i}(\mod Y^\circ_{\le i})}.
\end{align*} Suppose now that $a_r\equiv 1-\frac{1}{\gamma_i}x_{J_i}(\mod Y^\circ_{\le i})$ for some $r\ge  1$. Same as the previous calculation, we have
\begin{align*}
a_{r+1}&\equiv {\textstyle3(1-\frac{1}{\gamma_i}x_{J_i})^2-2(1-\frac{1}{\gamma_i}x_{J_i})^3(\mod Y^\circ_{\le i})}\\
&={\textstyle1-\frac{1}{\gamma_i}x_{J_i}}.
\end{align*} As such, $f_{i+1}'=a_\infty$ has the desired form. The proof is now complete.
\end{proof}

We are now ready to prove the main result in this section.

\begin{theorem}\label{T: idem} The set $\{e_J:J\in\R_q\}$ is a complete set of orthogonal primitive idempotents of $\Des_\k$ such that $\theta_\k(e_J)=\Char_{J,\k}$ and $\sum_{J\in \R_q}e_J=1$. Furthermore, \[e_J=\textstyle\frac{1}{\gamma_J}x_J+\epsilon_J\] where $\epsilon_J$ is a linear combination of some $x_K$ such that $K\subsetneq_W J$, i.e., $\epsilon_J\in Y^\circ_J$ with coefficients belonging to the prime subfield of $\k$.
\end{theorem}
\begin{proof} \Cref{L: f'} and \cite[Proposition 3.3(ii)]{Lim:2023a} give us the first sentence of our theorem. We now prove inductively on $i$ that $e_J$ has the desired form. When $i=1$, since $Y^\circ_{\le 1}=Y^\circ_{J_1}$, we have \[\textstyle e_{J_1}=f_1'-f_2'\equiv 1-(1-\frac{1}{\gamma_1}x_{J_1})=\frac{1}{\gamma_1}x_{J_1}(\mod Y^\circ_{J_1}).\] For $i\in [2,m]$, as $Y^\circ_{\le i-1}\subseteq Y_{\le i-1}\subseteq Y^\circ_{\le i}$, by \Cref{L: f'}(ii), we have \[\textstyle e_{J_i}=-\left (\frac{1}{\gamma_{i-1}}x_{J_{i-1}}+\epsilon_{i-1}\right )+\left (\frac{1}{\gamma_i}x_{J_i}+\epsilon_i\right )\equiv \frac{1}{\gamma_i}x_{J_i}(\mod Y^\circ_{\le i}).\] To complete the proof, we are left to check that $e_{J_i}\in Y_{J_i}$ so that $e_{J_i}-\frac{1}{\gamma_i}x_{J_i}\in Y^\circ_{\le i}\cap Y_{J_i}=Y^\circ_{J_i}$. When $i=1$, $e_{J_1}=\frac{1}{\gamma_1}x_{J_1}+\epsilon_1\in Y_{\le 1}=Y_{J_1}$. Suppose now that $i\in [2,m]$. We have
\begin{align*}
f_{i+1}'&=(f_i'f_{\ge  i+1}f_i')_\infty=(f_i'(f_{\ge  i}-f_i)f_i')_\infty\equiv (f_i'f_{\ge  i}f_i')_\infty (\mod Y_{J_i})\\
&=((f_{i-1}'f_{\ge  i}f_{i-1}')_\infty f_{\ge  i}(f_{i-1}'f_{\ge  i}f_{i-1}')_\infty)_\infty\\
&=((f_{i-1}'f_{\ge  i}f_{i-1}')_\infty f_{i-1}'f_{\ge  i}f_{i-1}'(f_{i-1}'f_{\ge  i}f_{i-1}')_\infty)_\infty\\
&=(f_{i-1}'f_{\ge  i}f_{i-1}'(f_{i-1}'f_{\ge  i}f_{i-1}')_\infty)_\infty\\
&=(f_{i-1}'f_{\ge  i}f_{i-1}')_\infty=f_i'
\end{align*} where the equivalence is obtained using $f_i'f_if_i'\in Y_{J_i}$ and the construction in \Cref{P: lift idem}, the third equality follows from \Cref{L: f'}(i) and the fourth and fifth equalities follow from \Cref{P: lift idem}. As such, we get $e_{J_i}=f_{i+1}'- f_i'\equiv 0(\mod Y_{J_i})$.

The last assertion regarding the coefficients follows from our construction throughout this section. The proof is now complete.
\end{proof}


Recall the simple module $S_{J,\k}$ and the functor $\D(-)=\Hom_\k(-,\k)$ in Section \ref{S:Des}. We have the following corollary.

\begin{corollary}\label{C: proj inj cover} Let $J\in \R_q$. Then $P_{J,\k}\cong \Des_\k e_J$ and  $I_{J,\k}\cong \D(e_J\Des_\k)$.
\end{corollary}
\begin{proof} Use \Cref{T: idem} and \Cref{P: proj cover}.
\end{proof}

\section{Sufficient conditions for the descent algebras}\label{S: proj}

Recall that $\cO$ is a local principal ideal domain with maximal ideal $(\pi)$, $\K$ is the field of fractions of $\cO$ and $\F=\cO/(\pi)$. Let $p$ be the characteristic of $\F$. In this section, we show that the descent algebras $\Des_\cO$ of $W$ over $\cO$ satisfy \Cref{hyp:rad} if $p\nmid |W|$ and \Cref{hyp:freeJ} if $p\nmid n_W$ (see \Cref{D: dW}).

Similar as before, we write $\OD=\Des_\cO$, $\KD=\K\otimes_\cO \Des_\cO$, $\FD=\F\otimes_\cO \Des_\cO$ and, for any $\OD$-module $M$ that is $\cO$-free, we write $\hat M=\K\otimes_\cO M$ and $\bar M=\F\otimes_\cO M$. For each $z\in\OD$, the coefficients of the $x_J$'s in $z$ belongs in $\cO$ and we write $\bar z$ for the element in $\FD$ by reducing the coefficients modulo $\pi$. Furthermore, since we will be using the elements $f_i'$, $f_J$ and $e_J$ in \Cref{S:Idem} over different fields $\K$ and $\F$, we use $f_{i,\K}'$ and so on to emphasise the field which we work over.

\subsection{Verifying \Cref{hyp:rad}} In this subsection, we verify that the descent algebras satisfy \Cref{hyp:rad} if $p\nmid |W|$. For the next lemma, we note that if $s\in \cO\backslash (\pi)$, then $s^{-1}\in \cO$. As such, an element in $\K$ belongs in $\cO$ if it can be written in the form $rs^{-1}$ where $r,s\in\cO$ and $s\not\in(\pi)$.

\begin{lemma}\label{L: B0Bp} Suppose that $p\nmid |W|$. We have that $\R=\R_p=\R_0$ and the entries in $M_\K^{-1}$ belongs in $\cO$. Moreover, the matrix $M_{\F}^{-1}$ is obtained from $M_\K^{-1}$ by taking modulo $\pi$ on its entries. In particular, for each $J\in\R$, we have $f_{J,\K}\in\OD$ and $\bar f_{J,\K}=f_{J,\F}$.
\end{lemma}
\begin{proof} Since $p\nmid |W|$, we have $\beta_{JJ}=[\N_W(W_J):W_J]\not\in(\pi)$ for all $J\in\R$ and hence $\R=\R_p=\R_0$. Notice that $\beta_{JK}\in\cO$ for each $J,K\in\R$ and $M_\F=(\overline \beta_{JK})$. The adjugate matrix of $M_\K$ have entries belonging in $\cO$ and the adjugate matrix of $M_\F$ is obtained from that of $M$ modulo $\pi$. Furthermore, $\det(M_\K)=\prod_{J\in \R}\beta_{JJ}\not\in(\pi)$ and hence $\det(M_\K)^{-1}\in\cO$. Furthermore, $\det(M_\F)=\overline{\det(M)}\neq 0$. As such, the entries of $M_\K^{-1}$ belong in $\cO$ and they give the entries for $M_{\F}^{-1}$ upon modulo $\pi$.
\end{proof}

The proof of the following lemma requires the construction presented in \Cref{S:Idem} which we shall not go into the details.

\begin{lemma}\label{le: idemp e reducible} Suppose that $p\nmid |W|$. For each $J\in \R$, we have $e_{J,\K}\in\OD$  and $\bar e_{J,\K}=e_{J,\F}$. In other words, the complete set of primitive orthogonal idempotents $\{e_{J,\F}:J\in\R\}$ of $\FD$ are liftable to $\OD$.
\end{lemma}
\begin{proof} For each $J\in\R$, by  \Cref{L: B0Bp}, we have $f_{J,\K}\in\OD$ and $\bar f_{J,\K}=f_{J,\F}$. Since the structure constants $a_{IKL}$'s belong in $\cO$ and the construction of the idempotents are obtained by iterating the process $a_{r+1}=3a_r^2-2a_r^3$, and it stabilises after finite number of steps, we have that $f_{i,\K}'\in\OD$ and hence $e_{J,\K}\in\OD$. The construction `commutes' with the action of taking modulo $\pi$. As such, $\bar e_{J,\K}=e_{J,\F}$.
\end{proof}


\begin{corollary}\label{C: hypbasic for Des} Suppose that $p\nmid |W|$. We have \[\OD=\Sp_{\cO}\{e_{J,\K}:J\in\R\}\oplus (\Rad(\KD)\cap \OD)\] and $\Rad(\OD)=\pi\OD+(\Rad(\KD)\cap \OD)$.
\end{corollary}
\begin{proof}
By \Cref{le: idemp e reducible}, $e_{J,\K}\in\OD$ for any
$J\in\R$. Since $e_{J,\K}$ are idempotents, the sum in the right hand
side of the display equation in the statement is clearly a direct sum. Let
$U:=\Sp_{\cO}\{e_{J,\K}:J\subseteq S\}\oplus (\Rad(\KD)\cap
\OD)$. Clearly, $U\subseteq \OD$. For the reverse inclusion, we
argue by induction on the total order $(\Rleq, \R)$. When
$J=\emptyset\in\R$, we have
$e_{\emptyset,\K}=\frac{1}{\gamma_\emptyset}x_\emptyset$ where
$\frac{1}{\gamma_\emptyset}\in \cO$. Thus $x_\emptyset\in U$. For any
$J\in\R$, by induction, we have
\[x_J=\gamma_Je_{J,\K}+\gamma_J(\text{$\cO$-linear combination of
    $x_K$ where $K\subsetneq_W J$})\in U.\]
For any $J'\subseteq S$ such that $J'=_WJ$, we have
$x_{J'}=x_J+(x_{J'}-x_J)\in U$ as $x_{J'}-x_J\in \Rad(\KD)\cap
\OD$. Therefore $\OD=U$. The expression for $\Rad(\OD)$ now follows from the decomposition.
\end{proof}

\begin{theorem}\label{T:weakdescent} Suppose that $p\nmid |W|$. The descent algebra $\OD$ satisfies \Cref{hyp:rad}. In particular, \Cref{th:main}(ii) applies for $\OD$.
\end{theorem}
\begin{proof} The first statement follows by applying both \Cref{le: idemp e reducible} and \Cref{C: hypbasic for Des}. Therefore, $\OD$ satisfies \Cref{hyp:Ext} using \Cref{th:main}(i). Hence \Cref{th:main}(ii) applies for $\OD$.
\end{proof}

As a special case of \Cref{T:weakdescent}, we have the following corollary.

\begin{corollary}\label{C:descentExtQ} Let $W$ be a finite Coxeter group and suppose that $p\nmid |W|$. Then the Ext quivers of $\KD$ and $\FD$ are identical.
\end{corollary}

We remark that Schocker \cite{Schocker:2004a} computed the Ext quivers for the descent algebras of type $\mathbb{A}$ over any field of characteristic zero. Saliola \cite{Saliola:2008a} extended the computation to both types $\mathbb{A}$ and $\mathbb{B}$ over $k$ given that $p\nmid |W|$ and particularly showed that the Ext quivers of $\KD$ and $\FD$ are identical. Our corollary asserts that such phenomenon holds for arbitrary Coxeter group. In particular, \Cref{C:descentExtQ} and Schocker's result give an alternative proof to Saliola's result in the type $\mathbb{A}$ case.


\subsection{Verifying \Cref{hyp:freeJ}} Recall the set $\I(n)$ we have introduced in Section \ref{S:Des}. For each $J\in\R$, we fix $J_1,\ldots,J_{n_J}$ the distinct subsets of $S$ such that $J_i=_WJ$. Let
\begin{align}\label{Eq: Basis B}
\B&=\{(J_i,J_{i+1}):J\in\R,\ i\in [1,n_J-1]\}\subseteq \I(1),\\
\Ba&=\{x_{J_i}-x_{J_{i+1}}:(J_i,J_{i+1})\in \B\}.\nonumber
\end{align}
Notice that, for any $1\le i<j\le n_J$, we have $x_{J_i}-x_{J_j}=\sum^{j-1}_{k=i}(x_{J_k}-x_{J_{k+1}})$. Notice also that $\Ba$ forms a basis for $\Rad(\KD)$ (respectively, for $\Rad(\FD)$ when $p\nmid |W|$) (cf. \Cref{L: basis for J^r}).



\begin{definition} We fix a total order for $\I(n)$ (so that $\B$ inherits that of $\I(1)$). Define the matrix $[\I(n)]$ where the rows and columns of $[\I(n)]$ are labelled by $\I(n)$ and $\B$ respectively and the corresponding entry is the coefficient of $x_L-x_{L'}$ in $x_{\un{J},\un{K}}$ as elements in $\OD$ where $(\un{J},\un{K})\in \I(n)$ and $(L,L')\in \B$. Let $[\I(n)]_\K=\K\otimes_\cO [\I(n)]$ and $[\I(n)]_\F=\F\otimes_\cO [\I(n)]$.
\end{definition}

\begin{lemma}\label{L: rank} Let $(W,S)$ be a Coxeter system and $n$ be a positive integer.
\begin{enumerate}[\rm (i)]
\item The entries of $[\I(n)]$ belong in $\Z$.
\item Suppose further that $p\nmid |W|$. Then $\rank([\I(n)]_\K)=\dim_\K \Rad^n(\KD)$ and $\rank([\I(n)]_\F)=\dim_\F \Rad^n(\FD)$.
\end{enumerate}
\end{lemma}
\begin{proof} Part (ii) follows from \Cref{L: basis for J^r}. We now prove part (i). For each $(\un{J},\un{K})\in \I(n)$ and $L\subseteq S$, since the structure constants of the descent algebra belong in $\Z$, we have $x_{\un{J},\un{K}}\in\Des_\Z$. Also, since $x_{\un{J},\un{K}}\in\Rad(\Des_\Q)=\Sp_\Q \Ba$, the sum of all the coefficients in terms of the basis $\{x_J:J\subseteq S\}$ is $0$ and hence the entries in $[\I(n)]$ belong in $\Z$.
\end{proof}

\begin{definition}\label{D: dW} Let $(W,S)$ be a Coxeter system and $s=\dim_\K \Rad^n(\KD)\neq 0$ (or equivalently, $[\I(n)]$ is not a zero matrix). Let $d_{W,n}$ be the greatest common divisor of the determinants of all possible $(s\times s)$-submatrices of $[\I(n)]$. Furthermore, let $n_{W,n}=\lcm(d_{W,n},|W|)$ and $n_W$ be the least common multiple of the $n_{W,n}$'s such that $d_{W,n}\neq 0$, i.e., \[n_W=\lcm\{n_{W,n}:n\le \ell\}\] where $\ell$ is the radical length of $\KD$ (that is, the largest non-negative integer $\ell$ such that $\Rad^\ell(\KD)\neq 0$).
\end{definition}

Notice that, if $\B\neq\emptyset$, then $[\I(1)]$ contains an identity submatrix and hence $d_{W,1}=1$. At the moment, the number $d_{W,n}$ seemed to depend on the choice of $\B$ (and hence the basis $\Ba$). In Section \ref{S: indep}, we will show that it is not the case if we replace $\B$ by another basis $\B'$ within $\I(1)$.

Before proceeding to the next lemma, we provide some examples.

\begin{example} Consider the Coxeter group $(W,S)$ of the dihedral type $\mathbb{I}_n$ where $S=\{s_1,s_2\}$ so that $W$ is the dihedral group of order $2n$. It is well-known that $\{s_1\}=_W\{s_2\}$ if and only if $n$ is odd. Therefore,  $\B=\emptyset$ unless $n$ is odd. In the case when $n$ is odd, $B=\{x_{\{s_1\}}-x_{\{s_2\}}\}$ and $[\I(r)]$ is zero for $r\ge  2$. Therefore, for each $n$, we have $n_W=2n$.
\end{example}

\begin{example} Consider the Coxeter group of $(W,S)$ type $\mathbb{A}_3$ where $S=\{s_1,s_2,s_3\}$. We have \[\B=\{x_{\{s_1\}}-x_{\{s_2\}},x_{\{s_2\}}-x_{\{s_3\}},x_{\{s_1,s_2\}}-x_{\{s_2,s_3\}}\}.\] The matrix $[\I(r)]$ is zero for $r\ge  3$. For $r=2$, $\Rad^2(\KD)$ is one-dimensional. Since \[(x_{\{s_1,s_2\}}-x_{\{s_2,s_3\}})^2=x_{\{s_1\}}-2x_{\{s_2\}}+x_{\{s_3\}},\] the corresponding row of $[\I(2)]$ is $\begin{pmatrix} 1&-1&0
\end{pmatrix}$. Therefore, $d_{W,2}=1$. Thus $n_W=24$.
\end{example}

\begin{example} By Magma~\cite{Bosma/Cannon/Playoust:1997a}
  computation, for type $\mathbb{E}_6$ or $\mathbb{E}_7$ and for each $n\in\{2,3,4\}$,
  we found a submatrix of $[\I(n)]$ with the prime factors of whose
  determinant are only $2,3,5$. For example, in the case of $\mathbb{E}_7$, one
  of the determinants is the following 53-digit number: \[17617366159747974325631727124547934643940229120000000=2^{105}3^{33}5^7.\]
\end{example}

\begin{question} Let $(W,S)$ be a Coxeter system and $p$ be a prime number. By definition, $p\nmid n_W$ implies $p\nmid |W|$. Is the converse true?
\end{question}

\begin{lemma}\label{L: rth radical basis} Let $(W,S)$ be a Coxeter system, $n$ a positive integer and suppose that $p\nmid n_{W,n}$. We have $\R_p=\R$ and there exists a subset $B(n)\subseteq \I(n)$ such that the sets $\{x_{\un{J},\un{K}}:(\un{J},\un{K})\in B(n)\}\subseteq \KD$ and $\{\bar x_{\un{J},\un{K}}:(\un{J},\un{K})\in B(n)\}\subseteq \FD$ form bases for $\Rad^n(\KD)$ and $\Rad^n(\FD)$ respectively. In particular, $\dim_\K \Rad^n(\KD)=\dim_\F \Rad^n(\FD)$.
\end{lemma}
\begin{proof} Since $p\nmid |W|$, we have $\R_p=\R$. By \Cref{L: basis for J^r}, there is a subset $B(n)$ of $\I(n)$ such that $T:=\{\bar x_{\un{J},\un{K}}:(\un{J},\un{K})\in B(n)\}$ forms a basis for $\Rad^n(\FD)$. Let $[B(n)]_\F$ be the submatrix of $[\I(n)]_\F$ labelled by the rows $B(n)$. Let $s=|B(n)|$. There is a $(s\times s)$-submatrix of $[B(n)]_\F$ with non-zero determinant. Therefore, the corresponding submatrix of $[B(n)]$ has also non-zero determinant. As such, by \Cref{L: rank}(ii), we get \[\dim_\K \Rad^n(\KD)=\rank([\I(n)]_\K)\ge  \rank([B(n)]_\K)=s=\rank([\I(n)]_{\F})=\dim_\F \Rad^n(\FD).\]

Conversely, let $s'=\dim_\K \Rad^n(\KD)=\rank([\I(n)]_\K)$. By the assumption, $p\nmid d_{W,n}$, there exists a submatrix $D$ of $[\I(n)]_\K$ such that $\det(D)\not\equiv 0(\mod \pi)$. Since $D$ has entries belonging in $\cO$, taking modulo $\pi$, we obtain the corresponding submatrix $D_\F$ of $[\I(n)]_\F$. As such, $\det(D_\F)\neq 0$. We get $\dim_\F \Rad^n(\FD)\ge  s'$. Since $s=s'$, $B(n)$ is a desired subset.
\end{proof}

\begin{theorem}\label{th:descent}
Let $W$ be a finite Coxeter group and suppose that $p\nmid n_W$. Then $\OD$ satisfies \Cref{hyp:freeJ}. In particular, \Cref{th:main}(ii) and (iii) apply for $\OD$.
\end{theorem}

\begin{proof} Clearly, $\OD$ is $\cO$-free of finite rank. By \Cref{L: rth radical basis}, we have $\dim_\K \Rad^n(\KD)=\dim_\F \Rad^n(\FD)$ for all $n\ge  0$. Using \Cref{le:Ofree-equiv} ((v) $\Rightarrow$ (i)), we obtain that $\OD$ satisfies our strongest hypothesis \Cref{hyp:freeJ}. As such, by \Cref{th:main}(i), the descent algebra satisfies the remaining hypotheses in \Cref{S:Intro}. In particular,  \Cref{th:main}(ii) and (iii) apply for $\OD$.
\end{proof}

\section{Independence of choice of basis for $n_W$}\label{S: indep}

In \Cref{D: dW}, we defined the numbers $d_{W,n}$ and $n_W$ based on
the choice of the basis $\B$ within $\I(1)$ (see \Cref{Eq: Basis
  B}). In this section, we address the question whether these numbers
are independent of the choice $\B$. The answer is positive and it is
given in \Cref{C: indep}. We shall begin with the discussion on
totally unimodular matrices.

A matrix $A$ over $\Z$ is said to have the consecutive-ones property
if $A$ is a 0-1 matrix (that is, its entries can only take values 0 or
1) and there is a  rearrangement of the columns of $A$ so that, in
each row, the ones appear consecutively. It is known that such a matrix
is totally unimodular, that is, every square non-singular submatrix of
$A$ has determinant $\pm1$ (see \cite{Fulkerson/Gross:1965a}). It follows that we have
the following lemma.

\begin{lemma}\label{L: consec 1} Suppose that $A$ is a square non-singular matrix with the consecutive-ones property. Then every entry of $A^{-1}$ can only take values $0$, $1$ or $-1$.
\end{lemma}
\begin{proof} This follows from the totally unimodularity that the cofactor matrix of $A$ admits entries with values $0$, $1$ or $-1$. Furthermore, $\det(A)=\pm1$. Therefore, we get the desired property for $A^{-1}$.
\end{proof}

We say that a matrix $A$ over $\Z$ has the signed consecutive-ones property if there exists a diagonal matrix $D$ with diagonal entries $\pm1$ and a matrix $B$ with consecutive-ones property such that $A=DB$, that is, every row of $A$ has the consecutive-ones or consecutive-minus-ones property. Since $A^{-1}=B^{-1}D$, we get the following.

\begin{corollary}\label{C: signed consec} Suppose that $A$ is a square non-singular matrix with the signed consecutive-ones property. Then every entry of $A^{-1}$ can only take values $0$, $1$ or $-1$.
\end{corollary}

In view of \Cref{D: dW}, for each non-zero matrix $M$ over $\Z$ of size $(m\times n)$ and $s$ be a positive integer not more than the rank of $M$, we write $\gcd_s(M)$ for the greatest common divisor of the determinants of all the $(s\times s)$-submatrices of $M$.

\begin{lemma}\label{L: gcd div gcd} Let $M$ be a non-zero $(m\times n)$-matrix over $\Z$, $s\le \rank(M)$, $c_1,\ldots,c_t$ be distinct numbers in $[1,n]$ for some positive integer $t$ and $(b_{ij})$ be an $(t\times t)$-matrix over $\Z$. Suppose that $M'$ is the matrix obtained by replacing, for each $j\in [1,t]$ and $i\in [1,m]$, the $(i,c_j)$-entry of $M$ by $\sum_{k=1}^tM_{i,c_k}b_{kj}$, that is, $M'=MA$ where $A$ is obtained from the identity matrix of size $n$ by replacing its $(c_i,c_j)$-entry by $b_{ij}$.  Then $\gcd_s(M)\mid \gcd_s(M')$.
\end{lemma}
\begin{proof} Let $d=\gcd_s(M)$ and, for subsets $I\subseteq [1,m]$ and $J\subseteq [1,n]$ of size $s$, we write $M_{IJ}$ for the corresponding $(s\times s)$-submatrix of $M$. Similarly, for $M'$. Furthermore, let \[\Gamma=\{(I,J):I\subseteq [1,m],\ J\subseteq [1,n],\ |I|=s=|J|\}.\] Consider $(I,J)\in\Gamma$ and let $J\cap \{c_j:j \in [1,t]\}=\{j_1,\ldots,j_\ell\}$. By multilinear and alternating properties of determinant, we have
\begin{align*}
\det(M'_{IJ})&=\sum_{c_j\in J}\sum_{(k_1,\ldots,k_\ell)}b_{k_1j_1}\cdots b_{k_\ell j_\ell}\det(M_I(k_1,\ldots,k_\ell))
\end{align*} where the second sum is taken over all $(k_1,\ldots,k_\ell)\in [1,t]^\ell$ and $M_I(k_1,\ldots,k_\ell)$ is the matrix obtained from $M_{IJ}$ by replacing the columns correspond to $c_{j_1},\ldots,c_{j_\ell}$ by the columns correspond to $c_{k_1},\ldots,c_{k_\ell}$ in the submatrix $M_{I,[1,n]}$. Notice that \[\det(M_I(k_1,\ldots,k_\ell))=\left \{\begin{array}{ll} 0 &\text{if $k_r=k_{r'}$ for some $r\neq r'$,}\\ \det(M_{I,K})& \text{otherwise,}\end{array}\right .\] where $K$ is obtained from $J$ by removing $c_{j_1},\ldots,c_{j_\ell}$ and then adding the distinct $c_{k_1},\ldots,c_{k_\ell}$. As such, $d\mid \det(M'_{IJ})$ and hence $d\mid \gcd_s(M')$.
\end{proof}

We obtain the following corollary addressing the independence of $n_W$ of the choice of basis within $\I(1)$.

\begin{corollary}\label{C: indep} For each positive integer $n$, the number $d_{W,n}$ is independent of the basis within $\I(1)$ we have chosen in \Cref{D: dW}. In particular, the same holds true for $n_W$.
\end{corollary}
\begin{proof} By \Cref{L: rank}(i), we obtain that $M:=[\I(n)]$ has integer entries. Let $\B'$ be another subset of $\I(1)$ such that $\Ba'$ forms another basis of $\Rad(\Des_\Q)$ and let $M':=[\I'(n)]$ be the matrix with respect to $\Ba'$ which also has integer entries. For each $J\subseteq S$, recall from \Cref{S: proj}, $J_1,\ldots,J_{n_J}$ are the distinct subsets of $S$ such that $J_i=_W J$ for each $i\in[1,n_J]$. For each $x_{J_i,J_j}\in T$, we have \[x_{J_i,J_j}=\left \{\begin{array}{ll}x_{J_i,J_{i+1}}+\cdots+x_{J_{j-1},J_j}&\text{if $i<j$,}\\ -(x_{J_i,J_{i+1}}+\cdots+x_{J_{j-1},J_j})&\text{if $i>j$.}\end{array}\right .\] Therefore, the transition matrix $A$ from $\Ba$ to $\Ba'$ has the signed consecutive-ones property. By \Cref{C: signed consec}, $A^{-1}$ has entries in $\{0,1,-1\}$. Notice that we can rearrange the columns of $M$ and $M'$ so that $M=M'A$. Let $s$ be the rank of $\Q\otimes_\Z M$ (or $\Q\otimes_\Z M'$). Applying \Cref{L: gcd div gcd} twice, we have \[{\gcd}_s(M)\mid {\gcd}_s(M')\mid {\gcd}_s(M)\] which forces the equality ${\gcd}_s(M)= {\gcd}_s(M')$. Therefore $d_{W,n}$ does not depend on the choice of basis within $\I(1)$.
\end{proof}

\section{nilCoxeter Algebras satisfy \Cref{hyp:freeJ}}\label{S:nilCox}

In this section, we show that the nilCoxeter algebra satisfies our strongest hypothesis. The algebra originates in \cite[Theorem 3.4]{Bernstein/Gelfand/Gelfand:1973a}. Subsequently, for example, it has been used in the study of the Schubert polynomials in \cite{Fomin/Stanley:1994a} and studied for its categorification aspect in \cite{Khovanov:2001a}.

Fix a finite Coxeter system $(W,S)$. For any integral domain $R$, we define the nilCoxeter algebra $\Nil_R(W)$ as the $R$-algebra generated by $\{x_s:s\in S\}$ subject to the relations $x_s^2=0$ and \[\underbrace{x_sx_{s'}x_s\cdots}_{\text{$m$ factors}}=\underbrace{x_{s'}x_sx_{s'}\cdots}_{\text{$m$ factors}}\] for distinct $s,s'\in S$ and where $m$ is the order of $ss'$. As a consequence, $\Nil_R(W)$ has an $R$-basis $\{x_w:w\in W\}$ such that \[x_ux_v=\left \{\begin{array}{ll} x_{uv}&\text{if $\ell(uv)=\ell(u)+\ell(v)$,}\\ 0&\text{otherwise.}\end{array}\right .\] Since $x_w$ is nilpotent for any $1\neq w\in W$, it belongs in $\Rad(\Nil_\k(W))$. As such, $\Nil_\k(W)$ is basic with a single simple module. In general, the $n$th radical layer $\Rad^n(\Nil_\k(W))$ has a basis $\{x_w:\ell(w)\geq n\}$.

\begin{lemma}\label{L:RadCox} Let $A=\Nil_\k(W)$ and $T$ be the unique simple $A$-module. The projective cover $P_T$ of $T$ is the regular $A$-module and
\begin{align*}
  [\Rad^n(P_T)/\Rad^{n+1}(P_T):T]&=\dim_\k \Rad^n(A)/\Rad^{n+1}(A)=|\{w\in W:\ell(w)=n\}|.
\end{align*}
\end{lemma}

Similar as before, let $\ON:=\Nil_\cO(W)$, $\KN:=\Nil_\K(W)$, and $\FN:=\Nil_\F(W)$.

\begin{theorem}\label{T:nilCox} Let $W$ be a Coxeter group. The algebra $\ON$ satisfies \Cref{hyp:freeJ}.
\end{theorem}
\begin{proof} By \Cref{L:RadCox}, we get $\dim_\F \Rad^n(\FN)/\Rad^{n+1}(\FN)=\dim_\K \Rad^n(\KN)/\Rad^{n+1}(\KN)$. Using \Cref{le:Ofree-equiv} ((v) $\Rightarrow$ (i)), we obtain that $\ON$ satisfies our strongest hypothesis \Cref{hyp:freeJ}.
\end{proof}

Let $T$ be the unique simple $\Nil_\k(W)$-module. \Cref{T:nilCox,th:main} imply that both $\dim_\k\Ext^i_{\Nil_\k(W)}(T,T)$ and hence the Hilbert–Poincar\'{e} series for the Ext algebra $\Ext^*_{\Nil_\k(W)}(T,T)$ are independent of the field $\k$. In the recent article \cite{Benson:nilCoxeter}, the first author studied the type $\mathbb{A}$ case. He computed the Hilbert–Poincar\'{e} series and presented a set of generators and relations for the Ext algebra. We refer the readers to the article for the full statement.

\begin{theorem}[{[\cite[Theorem 1.1]{Benson:nilCoxeter}}] Suppose that $n\ge 2$. The Ext algebra $\Ext^*_{\Nil_\Z(\sym{n})}(\Z,\Z)$ has the Hilbert–Poincar\'{e} series \[\sum_{i=0}^\infty t^i\left (\mathsf{rank}_\Z\Ext^i_{\Nil_\Z(\sym{n})}(\Z,\Z)\right )=\frac{1}{(1-t)^{n-1}}.\]
\end{theorem}



\iffalse

\section{g{The face semigroup algebras satisfy \Cref{hyp:basic}}
{The face semigroup algebras satisfy Hypothesis \ref{hyp:basic}}}\label{S:facesemigroup}

In this section, we prove that the face semigroup algebras also
satisfy \Cref{hyp:basic}. The face semigroup algebras are defined by
hyperplane arrangements in real spaces. The algebras are closely
related to the descent algebras in the following sense as proved by
Bidigare (see \cite{Bidigare:1997a}). When $W$ is a reflection group,
the face semigroup algebra defined by the hyperplanes corresponding to
the root system of $W$ has a natural $W$-action. The fixed point space
is isomorphic to the opposite of the descent algebra of $W$. We first
give a brief account of the subject and prove that the face semigroup
algebras satisfy \Cref{hyp:basic}.

Consider a finite collection $\A$ of linear hyperplanes containing the
origin of $\Re^d$  for some positive integer $d$. For each $H\in\A$,
we fix the two open half-spaces of $\Re^d$ and denote them by $H^+$
and $H^-$. Also, we write $H^0=H$. A face $x$ of $\A$ is a non-empty
intersection of the form \[x=\bigcap_{H\in\A} H^{\sigma_H(x)}\] where
$\sigma_H(x)\in\{+,-,0\}$ for each $H\in\A$. The sequence
$\sigma(x)=(\sigma_H(x))_{H\in\A}$ is called the sign sequence of
$x$. The face is a chamber if $\sigma_H(x)\neq 0$ for all
$H\in\A$. Let $\Fa$ be the set consisting of all the faces of $\A$. On
the other hand, let $\Ed$ be the set consisting of the edges of $\A$,
that is, the intersections of hyperplanes in $\A$. We allow trivial
intersection so that $\Re^d\in\Ed$. The set $\Ed$ is partially ordered
$\Rleq$ by the reverse inclusion, i.e., $X\Rleq Y$ if and only if
$Y\subseteq X$. The map $\supp:\Fa\to \Ed$ given
by \[\supp(x)=\bigcap_{\substack{H\in\A,\\ \sigma_H(x)=0}}H\] is
surjective.

The set $\Fa$ is a semigroup with multiplication given by, for $x,y\in
\Fa$,
\[\sigma_H(xy)=\left \{\begin{array}{ll}
\sigma_H(x)&\text{if $\sigma_H(x)\neq 0$,}\\
\sigma_H(y)&\text{if $\sigma_H(x)=0$.}\end{array}\right .\]
The identity element $\ifa$ of $\Fa$ is the intersection of the
hyperplanes of $\A$ which has the sign sequence $\sigma_H(\ifa)=0$ for
any $H\in\A$. On the other hand, $\Ed$ is also a semigroup with
multiplication defined by, for $X,Y\in\Ed$, $X\vee Y$ is the sum of
the subspaces $X$ and $Y$, i.e., the small subspace of $\Re^d$
containing both $X$ and $Y$. The identity element of $\Ed$ is $\ied$
which is also the intersection of all hyperplanes of $\A$. Clearly,
$X\vee Y=Y$ if $Y\Rleq X$. We have $\supp(xy)=\supp(x)\vee \supp(y)$.

Let $R$ be an integral domain and $R\Fa,R\Ed$ be the semigroup
$R$-algebras. The algebra $R\Fa$ is called the face semigroup
algebra. The surjection $\supp:\Fa\to\Ed$ induces the surjective
$R$-algebra homomorphism $\supp:R\Fa\to R\Ed$.

\begin{theorem}[\cite{Bidigare:1997a}]\label{th:Bidigare} Let $k$ be a
  field. The algebra $\k\Fa$ is basic, $\Rad(\k\Fa)=\Ker(\supp)$ is
  spanned by \[\{x-x':X\in\Ed,\ x,x'\in\Fa\cap \supp^{-1}(X)\}\] and
  the irreducible representations of $\k\Fa$ are parametrised by the
  set $\Ed$.
\end{theorem}

In view of \Cref{th:Bidigare}, a complete set of primitive orthogonal
idempotents of $\k\Fa$ is parametrised by the set $\Ed$ which we shall
now describe. More generally, let
\[e_{\Re^d}=\bigcap_{H\in\Fa} H^+\in R\Fa.\]
For any $X\in\Ed$, let $\supp(x)=X$, define the element $e_X$
recursively by the formula
\[e_X=x-\sum_{Y<X} xe_Y\in R\Fa.\]

\begin{theorem}[\cite{Saliola:2009a}]\label{th:faceidem} Let $X\in\Ed$
  and suppose that $\supp(x)=X$. The set $\{e_X\}_{X\in\Ed}\subseteq
  \k\Fa$ is a complete set of primitive orthogonal idempotents of
  $\k\Fa$.
\end{theorem}

\KJ{As \cite[\S2]{Saliola:2009a}, for any two $X,Z\in \Fa$, we write $(X,Z)$ and $[X,Z]$ for the open and closed intervals respectively and let $\ell([X,Z])$ denote the length of $[X,Z]$.}

\KJ{\begin{lemma}[{\cite[Proposition 8.5]{Saliola:2009a}}]\label{L:semigroupRel} Let $Q$ be the Ext quiver of $\k\Fa$. We have $\k\Fa\cong \k Q/I$ where $I$ is the ideal generated by, for each $X,Z\in Q$ such that $\ell([X,Z])=2$, the relations \[\sum_{Y\in (X,Z)} (X\to Y\to Z).\]
\end{lemma}}

Recall that $\cO$ is a local principal ideal domain with maximal ideal $(\pi)$, $\K$ is the field of fractions of $\cO$ and $\F=\cO/(\pi)$.

\begin{lemma}\label{le:semigroupidem}\
\begin{enumerate}[\rm (i)]
\item The complete set of primitive orthogonal idempotents $\{e_X:X\in\Ed\}$ of $\cO\Fa$ reduces to a complete set of primitive orthogonal idempotents $\{\bar e_X:X\in\Ed\}$ of $\F\Fa$.
\item For $X\in \Ed$, over $\cO$, we have $\supp(e_X)=X+\sum_{Y<X}\lambda_Y Y$ for some $\lambda_Y\in \Z$.
\end{enumerate}
\end{lemma}
\begin{proof} Part (i) is clear using \Cref{th:faceidem} and the definition of $e_X$. For part (ii), we argue by induction on the partial order $\Rleq$ of $\Ed$. When $X=\Re^d$, $\supp(e_X)=\Re^d$. Let $X,Y\in\Ed$, $Y<X$ and suppose that $\supp(e_Y)=Y+\sum_{Z<Y}\lambda_ZZ$ where $\lambda_Z\in\Z$. We have \[\supp(xe_Y)=X\vee Y+\sum_{Z<Y}\lambda_Z(X\vee Z)=Y+\sum_{Z<Y}\lambda_ZZ=\supp(e_Y).\] So $\supp(e_X)=X-\sum_{Y<X}\supp(xe_Y)$ has the desired property.
\end{proof}

\begin{lemma}\label{L:semigroupRad} \KJ{We have $\cO\Fa=\Sp_\cO\{e_X:X\in \Ed\}\oplus(\Rad(\K\Fa)\cap \cO\Fa)$ and $\Rad(\cO\Fa)=\pi \cO\Fa+(\Rad(\K\Fa)\cap \cO\Fa)$. }
\end{lemma}
\begin{proof}
Let $U=\Sp_\cO\{e_X:X\in \Ed\}+(\Rad(\K\Fa)\cap \cO\Fa)\subseteq
  \cO\Fa$. Clearly, the sum is direct and $U\subseteq \cO\Fa$. For the reverse
inclusion $\cO\Fa\subseteq U$, we argue by induction on the partial
order $\Rleq$ of $\Ed$. By definition, $x:=\bigcap_{H\in\Fa}
H^+=e_{\hat 0}\in U$. Notice that $\supp(x)=\Re^d$. For any $x'\in\Fa$
such that $\supp(x')=\Re^d$, we have $x'=x+(x'-x)\in U$. Let $X\in\Ed$
and suppose that, for any $Y<X$ and $\supp(y)=Y$, we have $y\in U$. As
such, for each $Y<X$, by \Cref{le:semigroupidem}(ii), $\supp(xe_Y)$ is
a linear combination of $Z$ such that $Z\Rleq Y<X$ and hence $xe_Y$
belongs in $U$. Therefore,
\[x=e_X-\sum_{Y<X}xe_Y\in U.\]
For $x'\in\Fa$ such that $\supp(x')=X$, we argue as before that $x'\in
U$. The expression for $\Rad(\cO\Fa)$ follows from the decomposition of $\cO\Fa$.
\end{proof}

We now describe the simple modules for the face semigroup
algebras. Let $X\in\Fa$ and $S_X$ be the $\cO$-free module of rank $1$
with the action of $\cO\Fa$ defined by the character
\[\chi_X(y)=\left\{\begin{array}{ll} 1&\supp(y)\Rleq X,\\
0&\text{otherwise.}\end{array}\right .\]
Then $\hat S_X$ and $\bar S_X$ are irreducible modules for $\K\Fa$ and
$\F\Fa$ respectively by \cite{Brown:2000a}.


\begin{theorem}\label{th:face}
The face semigroup algebra $\cO\Fa$ satisfies \Cref{hyp:basic}. In
particular, \Cref{th:basic} holds for $\cO\Fa$ and, for $X,Y\in\Ed$,
we have, for all $t\ge  0$, $\Ext^t_{\cO\Fa}(S_X,S_Y)$ is $\cO$-free
and furthermore,
\begin{align*}
\Ext^t_{\F\Fa}(\bar S_X,\bar S_Y)&\cong \F\otimes_\cO\Ext^t_{\cO\Fa}(S_X,S_Y),\\
\Ext^t_{\K\Fa}(\hat S_X,\hat S_Y)&\cong \K\otimes_\cO\Ext^t_{\cO\Fa}(S_X,S_Y).
\end{align*}
\end{theorem}
\begin{proof}  \KJ{By \Cref{le:semigroupidem,L:semigroupRad}, the face semigroup algebra satisfies \Cref{hyp:rad}. Let $Q$ be the Ext quiver of $\F\Fa$ (or $\K\Fa$). Consider the ideal $I$ of $\cO Q$ generated by \[\sum_{Y\in (X,Z)} (X\to Y\to Z)\in\cO Q\] so that, by \Cref{L:semigroupRel}, $\hat A\cong \K\Fa$ and $\bar A\cong \F\Fa$ where $A=\cO Q/I$. $\JQ^n=0$ for large $n$. Clearly, $I\le J_Q^2$. Furthermore, $\JA/\JA^2\cong \JQ/\JQ^2$ which is $\cO$-free. Thus $\cO\Fa$ satisfies \Cref{hyp:basic}. As such, by \Cref{th:main}, both \Cref{th:hyp=>free,th:KOF} apply.}
\end{proof}

\begin{example}
Let $H_1,\ldots,H_n$ be hyperplanes in $\Re^d$ and suppose that there
exists a subspace $\ell$ such that $H_i\cap H_j=\ell$ for all $1\le
i\neq j\le n$. By \cite[Proposition 8.5]{Saliola:2009a}, the Ext
quiver $Q$ of $\cO\Fa$ is given by
\[\xymatrix{&&\Re^d\ar[dll]_{\alpha_1}\ar[dl]^{\alpha_2}\ar[drr]^{\alpha_n}\ar[dr]_{\alpha_{n-1}}&&\\ H_1\ar[drr]_{\beta_1}&H_2\ar[dr]^{\beta_2}&{\hspace{.4cm}\cdots\hspace{.4cm}}&H_{n-1}\ar[dl]_{\beta_{n-1}}&H_n\ar[dll]^{\beta_n}\\ &&\ell&&}\] and $\cO\Fa\cong \cO Q/I$ where $I$ is the ideal generated by $\sum^n_{i=1}\beta_i\alpha_i$. In particular,
\begin{align*}
\dim_\K\K\Fa&=4n+1=\dim_\F\F\Fa,\\ \dim_\K\Rad(\K\Fa)&=3n-1=\dim_\F\Rad(\F\Fa),\\  \dim_\K\Rad^2(\K\Fa)&=n-1=\dim_\F\Rad^2(\F\Fa),\\
\dim_\K\Rad^r(\K\Fa)&=0=\dim_\F\Rad^r(\F\Fa)
\end{align*} for all $r\ge  3$. Thus, $\cO\Fa$ satisfies \Cref{hyp:freeJ} using \Cref{le:Ofree-equiv}.
\end{example}

\begin{question} In general, does $\cO\Fa$ satisfy \Cref{hyp:freeJ}?
\end{question}
\fi

\section{$\R$-trivial Monoid Algebras satisfy \Cref{hyp:rad}}\label{S:R-Trivial}

A (finite) monoid $M$ is $\R$-trivial if the subsets $\sigma M$ are all distinct when $\sigma$ runs through all elements of $M$, i.e., if $\sigma M=\tau M$ then $\sigma=\tau$. As shown by Schocker \cite{Schocker:2008a}, this class of algebras includes the unital left regular band algebras (see \Cref{S:LRegularBand}) and $0$-Hecke algebras. In \cite{Berg/Bergeron/Bhargava/Saliola:2011a}, Berg--Bergeron--Bhargava--Saliola proved that the notion of $\R$-trivial monoid is equivalent to weakly ordered monoid and gave a recursive formula to construct a complete set of primitive orthogonal idempotents of  $\C M$. Subsequently, for a general ring $R$ (with known complete system of primitive orthogonal idempotents, for example, $\Z$ or a field), a recursive formula has also been given by  Nijholt--Rink--Schwenker \cite{Nijholt/Rink/Schwenker:2021a} for $R M$ by using the observation that $\Z M$ is isomorphic to a subalgebra $\U_M^\Z$ of the upper triangular matrices $\U(\Z,n)$ where $n=|M|$. In particular, we have the following theorem.

\begin{theorem}[\cite{Nijholt/Rink/Schwenker:2021a}]\label{T:NRS} There are $m$ disjoint nonempty subsets $T_1,\ldots,T_m$ of $[1,n]$ such that $m$ is the number of simple modules for $\Z M$ and a complete set of primitive orthogonal idempotents $\{E_1,\ldots,E_m\}$ of $\U_M^\Z$ such that, for each $i\in [1,m]$ and $j\in [1,n]$, \[(E_i)_{jj}=\left \{\begin{array}{ll} 1&\text{if $j\in T_i$,}\\ 0&\text{if $j\not\in T_i$,}\end{array}\right .\] Furthermore, for arbitrary ring $R$, we have $R M\cong R\otimes_\Z \U^\Z_M=\U^R_M$ and, over a field $\k$, $\{\bar E_1,\ldots,\bar E_m\}$ is a complete set of primitive orthogonal idempotents of $\U^\k_M$.
\end{theorem}

Again, let $\cO$ be a local principal ideal domain with maximal ideal $(\pi)$, $\K$ be the field of fractions of $\cO$ and $\F=\cO/(\pi)$.

\begin{theorem}\label{T:R-Trivial} The $\R$-trivial monoid algebras satisfy \Cref{hyp:rad}. In particular, \Cref{th:main}(ii) applies for this class of algebras.
\end{theorem}
\begin{proof} We identify $\cO M$ with $\U^\cO_M$ under the isomorphism in \Cref{T:NRS}. The algebra $\cO M$ has finite $\cO$-rank as $M$ is finite. By \Cref{T:NRS}, the primitive orthogonal idempotents of $\U^\F_M$ are obtained from those in $\U^\Z_M$ (and hence $\U^\cO_M$) by taking entry-wise modulo $(\pi)$. In particular, they are liftable to $\U^\cO_M$. We are left to show that  \[\Rad(\U^\cO_M)=\pi \U^\cO_M+(\Rad(\U^\K_M)\cap \U^\cO_M).\] Let $\T$ and $\U^+$ be the diagonal and strictly upper triangular matrices in $\U(\cO,n)$ respectively. Clearly, $\Rad(\U^\cO_M)$ belongs in $\pi \T\oplus \U^+$. Let $C\in \Rad(\U^\cO_M)$. By \Cref{T:NRS}, there exist $c_1,\ldots,c_m\in (\pi)$ such that $C-\sum_{i=1}^mc_i E_i\in \U^+$. Since $\Rad(\U^\K_M)=\U^\K_M\cap \U^+$ and $\U^\K_M$ is obtained from $\U^\cO_M$ by extension of scalar entry-wise, we have $C-\sum_{i=1}^mc_i E_i\in \Rad(\U^\K_M)\cap \U^\cO_M$. Thus $\cO M$ satisfies \Cref{hyp:rad}. By \Cref{th:main}(i), $\cO M$ satisfies \Cref{hyp:Ext} and hence \Cref{th:main}(ii) applies for $\cO M$.
\end{proof}

\begin{example} Let $(W,S)$ be a Coxeter system as in \Cref{S:Des} and $S=\{s_i:i\in I\}$. The $0$-Hecke algebra $\Hec_R$ over an arbitrary commutative ring $R$ with 1 is the $R$-algebra with identity generated by $\{T_i:i\in I\}$ subject to the relations
\begin{enumerate}[\rm (i)]
\item $T_i^2=-T_i$ for all $i\in I$, and
\item for $i,j\in I$ with $i\neq j$, if $n_{ij}$ is the order of $s_is_j$, we have \[\underbrace{T_iT_jT_i\cdots}_{\text{$n_{ij}$ factors}}=\underbrace{T_jT_iT_j\cdots}_{\text{$n_{ij}$ factors}}.\]
\end{enumerate} Notice that $\Hec_\K\cong \K\otimes_\cO\Hec_\cO$ and $\Hec_\F\cong \F\otimes_\cO \Hec_\cO$. Norton \cite{Norton:1979a} showed that the $0$-Hecke algebras are basic over arbitrary field and their simple modules are parametrised by the subsets of $S$. For any $J\subseteq S$, let $S_J$ be the $\cO$-free module of rank $1$ defined by the character \[\lambda_J(T_i)=\left \{\begin{array}{ll} 0&\text{if $s_i\in J$,}\\ -1&\text{if $s_i\not\in J$.}\end{array}\right .\] As such, $\hat S_J:=\K\otimes_\cO S_J$ and $\bar S_J:=\F\otimes_\cO S_J$ are the simple modules for $\Hec_\K$ and $\Hec_\F$ respectively. Following Schocker \cite{Schocker:2008a}, the $0$-Hecke algebras are $\R$-trivial monoid algebras. As such, by \Cref{T:R-Trivial}, for $I,J\subseteq S$ and all $t\ge  0$, $\Ext^t_{\Hec_\cO}(S_I,S_J)$ is $\cO$-free
and furthermore,
\begin{align*}
\Ext^t_{\Hec_\F}(\bar S_I,\bar S_J)&\cong \F\otimes_\cO\Ext^t_{\Hec_\cO}(S_I,S_J),\\
\Ext^t_{\Hec_\K}(\hat S_I,\hat S_J)&\cong \K\otimes_\cO\Ext^t_{\Hec_\cO}(S_I,S_J).
\end{align*} In particular, $\dim_\F \Ext^t_{\Hec_\F}(\bar S_I,\bar S_J)=\dim_\K \Ext^t_{\Hec_\K}(\hat S_I,\hat S_J)$.
\end{example}

\section{Connected CW Left Regular Bands satisfy \Cref{hyp:basic}}\label{S:LRegularBand}

In this section, we consider the class of algebras $\k B$ where $B$ is a left regular band, i.e., $B$ is a semigroup satisfying the identities $x^2=x$ and $xyx=xy$ for all $x,y\in B$. Since $B$ may not have an identity, the algebra $\k B$ does not necessarily admit a unit.  However, following Margolis--Saliola--Steinberg \cite{Margolis/Saliola/Steinberg:2021a}, to guarantee a unit, it suffices to assume that $B$ is a connected CW left regular band -- we shall not need the full definition here and refer the reader to their paper. In the case $B$ is a monoid, it is also $\R$-trivial (see \Cref{S:R-Trivial}). This class of algebras includes the face algebras of hyperplane arrangements in real spaces. In the special case when the hyperplane arrangement is given by the reflection arrangement of a (finite) Coxeter group $W$, the descent algebra of $W$ is isomorphic to the opposite of the fixed point space of the face algebra under the action of $W$ as shown by Bidigare \cite{Bidigare:1997a}.

Throughout this section, we assume that $B$ is a connected CW left regular band. Let $\Lambda(B)=\{Bb:b\in B\}$ be the support semilattice of $B$ and \[\sigma:B\to \Lambda(B)\] be the support map, defined as $\sigma(a)=Ba$, which is a semigroup homomorphism. Recall that, given a graded (finite) poset $(\Rleq ,P)$, the closed interval $[X,Z]$ is the subset $\{Y:X\Rleq Y\Rleq Z\}$ of $P$. Similarly, we denote $(X,Z)$ for the open interval. The rank (or length) $\ell([X,Z])$ of $[X,Z]$ is the length of a longest chain starting at $X$ and ending at $Y$.

\begin{theorem}[{\cite[Theorem 1.1(4,5)]{Margolis/Saliola/Steinberg:2021a}}]\label{T:MSS} Let $Q$ be the Ext quiver of $\k B$. The support semilattice $\Lambda(B)$ is a graded poset and $Q$ is the Hasse diagram of $\Lambda(B)$. Furthermore, $Q$ is acyclic, $\k B$ is basic and $\k B\cong \k Q/I$ where $I$ is the ideal generated by, for any two vertices $X,Z$ in $Q$ such that $\ell([X,Z])=2$, the relations \[\sum_{Y\in (X,Z)} (X\to Y\to Z).\]
\end{theorem}

Same as before, let $\cO$ be a local principal ideal domain with maximal ideal $(\pi)$, $\K$ be the field of fractions of $\cO$ and $\F=\cO/(\pi)$.

\begin{theorem}\label{T:RegularBand} Let $B$ be a connected CW left regular band. The algebra $\OB$ satisfies \Cref{hyp:basic}. In particular, \Cref{th:main}(ii) applies for $\OB$.
\end{theorem}
\begin{proof} Let $Q$ be the Ext quiver of $\FB$ (or $\KB$). Consider the ideal $I$ of $\cO Q$ generated by \[\sum_{Y\in (X,Z)} (X\to Y\to Z)\in\cO Q\] and let $A=\cO Q/I\cong \OB$. By \Cref{T:MSS}, $\FB\cong \F\otimes_\cO A$, $\KB\cong \K\otimes_\cO A$ and $\JQ^n=0$ for some positive integer $n\ge  2$ as $Q$ is acyclic. Clearly, $I\le J_Q^2$. Thus $A$ satisfies \Cref{hyp:basic} and hence \Cref{hyp:rad} by \Cref{th:main}(i). Therefore \Cref{th:main}(ii) applies for $A$.
\end{proof}

At this point, it is not clear whether the algebra $\OB$ satisfies our strongest hypothesis \Cref{hyp:freeJ}. We consider an example.

\begin{example} Consider the face algebra $\cO\Fa$ with the hyperplane arrangement defined by the hyperplanes $H_1,\ldots,H_n$ in $\Re^d$. Suppose further that there
exists a subspace $X$ such that $H_i\cap H_j=X$ for all $1\le
i\neq j\le n$. By \cite[Proposition 8.5]{Saliola:2009a}, the Ext
quiver $Q$ of $\F\Fa$ (or $\K\Fa$) is given by the diagram below.
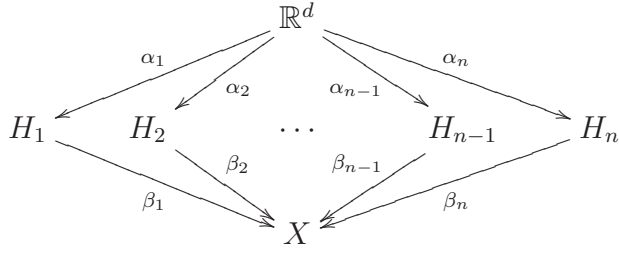
\[\xymatrix{&&\Re^d\ar[dll]_{\alpha_1}\ar[dl]^{\alpha_2}\ar[drr]^{\alpha_n}\ar[dr]_{\alpha_{n-1}}&&\\ H_1\ar[drr]_{\beta_1}&H_2\ar[dr]^{\beta_2}&{\hspace{.4cm}\cdots\hspace{.4cm}}&H_{n-1}\ar[dl]_{\beta_{n-1}}&H_n\ar[dll]^{\beta_n}\\ &&X&&}\]  The only relation generating the ideal $I$ is $\sum^n_{i=1}\beta_i\alpha_i$. As such,
\begin{align*}
\dim_\K\K\Fa&=4n+1=\dim_\F\F\Fa,\\ \dim_\K\Rad(\K\Fa)&=3n-1=\dim_\F\Rad(\F\Fa),\\  \dim_\K\Rad^2(\K\Fa)&=n-1=\dim_\F\Rad^2(\F\Fa),\\
\dim_\K\Rad^r(\K\Fa)&=0=\dim_\F\Rad^r(\F\Fa)
\end{align*} for all $r\ge  3$. Thus, $\cO\Fa$ satisfies \Cref{hyp:freeJ} using \Cref{le:Ofree-equiv} ((v) $\Rightarrow$ (i)). In particular, both \Cref{th:main}(ii) and (iii) apply for $\cO\Fa$.
\end{example}

\begin{question} Let $B$ be a connected CW left regular band. Does $\OB$ satisfy \Cref{hyp:freeJ}?
\end{question}



\bibliographystyle{amsplain}
\bibliography{repcoh}

\end{document}